\pdfoutput=1
\documentclass[11pt, reqno]{amsart}

    \advance \topmargin by -\headheight
    \advance \topmargin by -\headsep

    \evensidemargin \oddsidemargin
\usepackage{enumitem}   
\usepackage{graphicx}
\usepackage{amssymb}
\usepackage{mathrsfs}
\usepackage{epstopdf}

\usepackage{amsbsy,amssymb,amscd,amsfonts,latexsym,amstext,delarray, amsmath,graphicx,color,caption}
\usepackage{graphicx}
\usepackage{mathtools}
\usepackage{hyperref} 
\usepackage{amsmath}
\usepackage{amssymb}

\usepackage{mathtools}
\usepackage{epstopdf}
\usepackage[new]{old-arrows}

\usepackage[english]{babel}
\usepackage[autostyle]{csquotes}
\usepackage{ extpfeil}
\usepackage{relsize}

\usepackage{hyperref}
\usepackage[english]{babel}
\usepackage[autostyle]{csquotes}
\usepackage{textcomp}    
\usepackage{amscd}
\usepackage{tikz}
\usepackage{tikz-cd}
\usetikzlibrary{positioning}
\usetikzlibrary{arrows,arrows.meta}
\usepackage{extarrows}
\usepackage{tikz}
\usetikzlibrary{positioning}

\tikzset{>=stealth}

\usetikzlibrary{positioning}
\usepackage{extarrows}
\DeclareMathOperator{\Aut}{Aut}

 \hypersetup{
     colorlinks   = true,
     citecolor    = blue
}

\numberwithin{equation}{section}
\newtheorem{theorem}{Theorem}[subsection]
\newtheorem{proposition}[theorem]{Proposition}
\newtheorem{remark}[theorem]{Remark}
\newtheorem{corollary}[theorem]{Corollary}
\newtheorem{lemma}[theorem]{Lemma}

\newtheorem*{claim}{Claim}
\theoremstyle{definition}
\newtheorem{definition}[theorem]{Definition}

\title[GINDIKIN-KARPELEVICH FINITENESS]{ FINITENESS THEOREMS  FOR  KAC-MOODY GROUPS OVER NON-ARCHIMEDEAN LOCAL FIELDS}
\author{Abid Ali }
\address{Department of Mathematical and Statistical Sciences\\
University of Alberta\\
  632 Central Academic Building, Edmonton AB T6G 2G1, Canada}
\email{abid3@ualberta.ca}

\begin{document}
\begin{abstract}
We prove the finiteness of formal analogues of the spherical function (Spherical Finiteness), the ${\mathbf c}$-function  (Gindikin-Karpelevich Finiteness), and obtain a formal analogue of Harish-Chandra's limit (Approximation Theorem) relating spherical and ${\mathbf c}$-function in the setting of $p$-adic Kac-Moody groups.
The finiteness theorems imply that the formal analogue of the Gindikin-Karpelevich integral is well defined in local Kac-Moody settings. These results extend the Braverman-Garland-Kazhdan-Patnaik's affine Gindikin-Karpelevich finiteness theorems from \cite{BGKP} and provide an algebraic analogue of the geometrical results of Gaussent-Rousseau \cite{GR2} and H\'{e}bert \cite{Aug}. 

\end{abstract}
\maketitle

\section {Introduction }\label{Fin&Aff}
\subsection{} In \cite{HC1}, Harish-Chandra  associates a certain spherical function to a connected, semisimple Lie group $G$ with finite center. To describe this function, let $\mathfrak{g}:={\rm Lie}(G)$ and $K$ be a maximal compact subgroup of $G$, and $\theta: \mathfrak{g} \rightarrow \mathfrak{g}$ the corresponding Cartan involution, which has eigenvalues $\pm 1$. Letting $\mathfrak{k} = \{ X \in \mathfrak{g} \mid \theta(X) = X \}$ and $\mathfrak{p} = \{ X \in \mathfrak{g} \mid \theta(X) = - X \},$ we have a \emph{Cartan decomposition} $\mathfrak{g} = \mathfrak{k} \oplus \mathfrak{p}$. Pick $\mathfrak{a} \subset \mathfrak{p}$ a maximal, abelian subalgebra and let $A$ be the corresponding Lie subgroup of $G$; writing $\exp$ for the exponential map, we have $\exp(\mathfrak{a}) = A$. Under the adjoint action, the elements of $\mathfrak{a}$  are (jointly) diagonalizable and the non-zero eigenvalues are given by the roots of $\mathfrak{g}$. Pick a set of positive roots, and let $\mathfrak{n}^+$ denote the sum of the corresponding eigenspaces; it is a nilpotent subalgebra. Let $U^+$ be the corresponding Lie subgroup of $G$. The Iwasawa decomposition then states that 
\begin{eqnarray}
G=KAU^{+}=K\exp(\mathfrak{a})U^{+}
\end{eqnarray}
that is, each $g\in G$ can be written uniquely as $g=kau=k\exp(h)u$ for $k\in K$,$u\in U^{+}$, and $a=\exp(h)$, $h\in \mathfrak{a}$. Let $\mathfrak{a}_{\mathbb{C}}^{\ast}$ be the dual of the complexification $\mathfrak{a}_{\mathbb{C}}:= \mathbb{C} \otimes_{\mathbb{R} } \mathfrak{a}$. Let $\Delta_{0}\subset \mathfrak{a}_{\mathbb{C}}^{\ast}$ be the set of roots, $\Delta_{0,+}$ be the set of positive roots and $\Pi_{0}$ be set of simple roots. Let $A^{+}\subset A$ be the cone of dominant elements of $A$. Next, suppose $v\colon\mathfrak{a}\longrightarrow \mathbb{C}$ is a function. Associated with this function, we define a character $\phi_{v}\colon G\longrightarrow \mathbb{C}^{\ast}$ as $\phi_{v}(g)=e^{\langle v, h\rangle}$, where $g\in G$ has the Iwasawa decomposition $g=k\exp(h)u$ for some $k\in K$, $u\in U^{+}$ and $\exp(h)\in A$ with $h\in \mathfrak{a}$. Let $\rho=\frac{1}{2}\sum_{\alpha\in\Delta_{0,+}}\alpha$. A spherical function on $G$ associated with $v$ can be defined by the following integral (see corollary on page 61 in \cite{HC1})
\begin{eqnarray}\label{spfu}
f_{v}(g)&=&\int_{K}\phi_{v+\rho}(gk)dk,\; g\in G,
\end{eqnarray} 
where $dk$ is the normalized Haar measure on $K$.
Let $U^{-}$ be the unipotent group opposite to $U^{+}$. In Theorem 4 of {\it op. cit.}  the asymptotic behaviour of the spherical function $f_{v}$ was studied and Harish-Chandra showed
\begin{theorem}
Let $Re(iv)\in \mathfrak{a}^{\ast}$, then the limit ${\mathbf c}_{v}:=\lim_{a\overset{+}{\to}\infty}\frac{f_v(a)}{\phi_{v+\rho}(a)}$ exists and it is equal to
\begin{eqnarray}\label{GK1}
{\mathbf c}_{v}&=&\int_{U^{-}}\phi_{v+\rho}(u^{-})du^{-},
\end{eqnarray}
where $du^{-}$ is the Haar measure on $U^{-}$ normalized such that $\int_{U^{-}}\phi_{-2\rho}(u^{-})du^{-}=1$ and $a\overset{+}{\to}\infty$ means $a$ is made increasingly dominant in the cone of dominant elements $A^{+}$. 
\end{theorem}
The function ${\mathbf c}_{v}$ is known as the Harish-Chandra's ${\mathbf c}$-{\it function}.
 The integral on the right hand side of (\ref{GK1}) is called the Gindikin-Karpelevich integral so named after S. Gindikin and F. Karpelevic, who explicitly evaluated it in \cite{GiK}. Following Bhanu Murthy's inductive method from \cite{B}, a formula for the integral (\ref{GK1}) in terms of a product of certain gamma functions was obtained, which is now known as the {\it Gindikin-Karpelvich formula}. 

\subsection{}\label{sub3} 
There is also a non-archimedean analogue of the construction of the previous subsection studied by R. P. Langlands in his monograph \enquote{Euler Products} and I. G. Macdonald's notes \enquote{Spherical Functions on a Group of $p$-Adic Type}. To describe it, let $\mathcal{K}$ be a non-archimedean local field with ring of the integers $\mathcal{O}$. Pick $\pi$  a uniformizing element and $\mathrm{k}= \mathcal{O}/\pi\mathcal{O}$ be the finite residue field of cardinality $q$. Suppose $G=G(\mathcal{K})$ is a split, simply-connected Chevalley group over $\mathcal{K}$. We denote the integral subgroup $G(\mathcal{O})$ by $K$; this group is a non-archimedean analogue of maximal compact subgroup. Let $H$ be a Cartan subgroup and $H_{\mathcal{O}}=H\cap K$. The quotient group $A:=H/H_{\mathcal{O}}$ can be identified with the coweight lattice $\Lambda^{\vee}$ (which is equal to the coroot lattice $Q^{\vee}$ since $G$ is simply-connected) via the map $\mu^{\vee}\mapsto \pi^{\mu^{\vee}}$. The Iwasawa decomposition in this context states that $G=\cup_{\mu^{\vee}\in\Lambda^{\vee}}K\pi^{\mu^{\vee}}U^{+}$, that is, every $g\in G$ can be written as $g=k\pi^{\mu^{\vee}}u$ with $\mu^{\vee}$ uniquely determined by $g$ (note that $k\in K$ and $u\in U^{+}$ are not uniquely determined). For a map $v\colon A\longrightarrow \mathbb{C}^{\ast}$ we can define a function 
$$\phi_{v}\colon G\longrightarrow \mathbb{C}^{\ast}$$
as $\phi_{v}(g)=q^{ \langle v, \mu^{\vee}\rangle}$, if $g\in  G$ has Iwasawa decomposition as above $g\in K\pi^{\mu^{\vee}}U^{+}$, $\mu^{\vee}\in \Lambda^{\vee}$. The spherical function on $G$ can be defined as in (\ref{spfu})
\begin{eqnarray}
f_{v}(g)=&\int_{K}\phi_{v+\rho}(gk)dk,\label{padicintegra}
\end{eqnarray}
where $dk$ is the normalized Haar measure on $K$ so that $\int_{K}dk=1$. In an analogy with the archimedean case, by taking the limit $\lambda^{\vee}\rightarrow \infty$ in the dominant cone, the integral on the right hand side of (\ref{padicintegra}) over $K$ can be shifted to an integral over $U^{-}$, that is
\begin{eqnarray}\label{GK2}
{\bf c}_{v}:=\lim_{\lambda^{\vee}\overset{+}{\to}\infty}\frac{f_{v}(\pi^{\lambda^{\vee}})}{\phi_{v+\rho}(\pi^{\lambda^{\vee}})}&=&\int_{U^{-}}\phi_{v+\rho}(u^{-})du^{-},
\end{eqnarray}
where $du^{-}$ is an appropriately normalized Haar measure on $U^{-}$. Following the same inductive method of  Gindikin and Karpelevich, Macdonald in \cite[P. 77] {Mac} found the following formula for the above integral.
\begin{eqnarray}
{\bf c}_{v}&=&\prod _{\alpha^{\vee}\in \Delta_{0,+}^{\vee}}\frac{1-q^{-1-v( \alpha^{\vee})}}{1-q^{-v( \alpha^{\vee})}},
\end{eqnarray}
where $\Delta_{0,+}^{\vee}$ is the set of positive coroots. 
\subsection{}\label{sub4} 
Since Haar measures do not exist in general Kac-Moody groups over non-archimedean local fields, we need an algebraic formulation for the constructions given in the Subsection~\ref{sub3}. 
 Using the properties of $f_{v}$, Cartan and Iwasawa decompositions, one obtains the following formal version (i.e. valued in $\mathbb{C}[\Lambda^{\vee}]$ rather than $\mathbb{C}$) of the integral (\ref{padicintegra})
 
\begin{eqnarray}\label{stakemap}
\mathcal{S}_{\lambda^{\vee}}&:=&\sum_{\mu^{\vee}\in \Lambda^{\vee}}|K\backslash K\pi^{\mu^{\vee}}U^{+}\cap K\pi^{\lambda^{\vee}}K|q^{\langle \rho, \mu^{\vee}\rangle}e^{\mu^{\vee}},
\end{eqnarray}
such that the value of function $f_{v}(g)$ for $g\in K\pi^{\lambda^{\vee}}K$ satisfying $gk\in  K\pi^{\mu^{\vee}}U^{+}$  is equal to $f_{v}(g)= \mathrm{ev}_{v}(\mathcal{S}_{\lambda^{\vee}}),$ where $\mathrm{ev}_{v}\colon \mathbb{C}[\Lambda^{\vee}]\longrightarrow \mathbb{C}$ is the map defined as $e^{\mu^{\vee}}\mapsto q^{\langle v, \mu^{\vee}\rangle}$. Similarly, a formal version of the Gindikin-Karpelevich integral (\ref{GK2}) is written as:
\begin{eqnarray}
\mathscr{G}_{\lambda^{\vee}}&:=&\sum_{\mu^{\vee}\in \Lambda^{\vee}}|K\backslash K\pi^{\lambda^{\vee}-\mu^{\vee}}U^{+}\cap K\pi^{\lambda^{\vee}}U^{-}  |q^{\langle \rho, \lambda^{\vee}-\mu^{\vee}\rangle}e^{\lambda^{\vee}-\mu^{\vee}}.
\end{eqnarray}
Let us immediately notice the following \enquote{homogenity} property of the sum $\mathscr{G}_{\lambda^{\vee}}$, which shows that it suffices to obtain a formula for $\mathscr{G}_{0}$. 
\begin{lemma}\label{GKrem}
The sum $\mathscr{G}_{\lambda^{\vee}}$ satisfies
\begin{eqnarray}
\mathscr{G}_{\lambda^{\vee}}= q^{\langle \rho, \lambda^{\vee}\rangle}e^{\lambda^{\vee}}\mathscr{G}_{0}
\end{eqnarray}
\end{lemma}

The sum $\mathcal{S}_{\lambda^{\vee}}$ is connected with the {\it Satake map}, 
 $$\mathcal{S}:\mathscr{H}\longrightarrow \mathbb{C}[\Lambda^{\vee}]^{W},$$
 where the notation is as follows: $\mathscr{H}$ is the space of complex valued, compactly supported $K$-bi-invariant functions on $G$ with basis consisting of the characteristic functions $h_{\lambda^{\vee}}=\chi_{K\pi^{\lambda^{\vee}}K}$ of $K\pi^{\lambda^{\vee}}K$ for all $\lambda^{\vee}\in \Lambda^{\vee}_{+}$; $W$ is the Weyl group; $\mathbb{C}[\Lambda^{\vee}]$ is the group algebra of $\Lambda^{\vee}$; and $\mathbb{C}[\Lambda^{\vee}]^{W}$ is its $W$-invariant subspace. 
 
 For $\lambda^{\vee}\in \Lambda^{\vee}_{+}$, the Satake map $\mathcal{S}$ sends $h_{\lambda^{\vee}}$ to $\mathcal{S}_{\lambda^{\vee}}$. In \cite{Mac},  I. G. Macdonald determined an explicit formula for $\mathcal{S}_{\lambda^{\vee}}$
\begin{eqnarray}\label{satakeform}
\mathcal{S}_{\lambda^{\vee}}=\frac{q^{\langle \rho, \lambda^{\vee}\rangle}}{W_{\lambda^{\vee}}(q^{-1})}\sum_{w\in W}w(\Upsilon)e^{ w\lambda^{\vee}},
\end{eqnarray}
where $\Upsilon=\prod_{\alpha^{\vee}\in \Delta^{\vee}_{0,+}}\frac{1-q^{-1}e^{-\alpha^{\vee}}}{1-e^{-\alpha^{\vee}}}$ is a rational expression from $\mathbb{C}_{q}[\Lambda^{\vee}]:=\mathbb{C}[q, q^{-1}]\otimes_{\mathbb{C}}\mathbb{C}[\Lambda^{\vee}]$; and $W_{\lambda^{\vee}}(q^{-1})=\sum_{\sigma\in W_{\lambda^{\vee}}}q^{-\ell(\sigma)}$
 is the {\it Poincare polynomial} of the stabilizer $W_{\lambda^{\vee}}\subset W$ of $\lambda^{\vee}$, where $\ell\colon W\longrightarrow \mathbb{Z}$ denotes the length function on $W$. The formal analogue of the limit (\ref{GK2}) can be stated as:
\begin{theorem}[Approximation Theorem]\label{maintheorem1} 
For each $\mu^{\vee}\in \Lambda^{\vee}$, there exists $\lambda^{\vee}_{0}\in \Lambda^{\vee}_{+}$ regular such that for all $\lambda^{\vee}>\lambda^{\vee}_{0}$, we have
\begin{eqnarray}
K\pi^{\lambda^{\vee}-\mu^{\vee}}U^{+}\cap K\pi^{\lambda^{\vee}}U^{-}= K\pi^{\lambda^{\vee}-\mu^{\vee}}U^{+}\cap K\pi^{\lambda^{\vee}}K.
\end{eqnarray}
\end{theorem}
\noindent For finite dimensional groups, a proof of this result can be found in \cite[Proposition 3.6 (ii)]{BG}. By using Lemma~\ref{GKrem} and Theorem~\ref{maintheorem1}, $\mathscr{G}_{0}$ can be expressed as 
\begin{eqnarray}\label{Satakelimit}
\mathscr{G}_{0}=\lim_{\lambda^{\vee}\overset{+}{\to}\infty} \frac{\mathcal{S}_{\lambda^{\vee}}}{q^{\langle \rho, \lambda^{\vee}\rangle}e^{\lambda^{\vee}}},
\end{eqnarray}\label{GKfin23}
where $\lambda^{\vee}\overset{+}{\to}\infty$ indicates that the regular dominant coweight $\lambda^{\vee}$ is approaching to infinity while remaining within the regular dominant cone. By using the expression (\ref {satakeform}), one can compute the limit on the right hand side of (\ref{Satakelimit}) to obtain 
\begin{eqnarray}\label{Gkformulafin21}
\mathscr{G}_{0}=\Upsilon.
\end{eqnarray}
This strategy to compute the Gindikin-Karpelevich formula is stated and generalized for affine Kac-Moody groups in \cite{BGKP}.

\subsection{}\label{sub5} Suppose, now $G$ is a general Kac-Moody group over a non-archimedean local field $\mathcal{K}$. To consider $\mathscr{G}_{\lambda^{\vee}}$ and $\mathcal{S}_{\lambda^{\vee}}$ associated with $G$ and compute a formula for $\mathscr{G}_{\lambda^{\vee}}$, the first challenge is to show that these sums are well defined when $G$ is not of finite type, and an infinite dimensional version of Theorem~\ref{maintheorem1} holds. 

For $\mathscr{G}_{\lambda^{\vee}}$, one needs to prove the following,
\begin{theorem}[Gindikin-Karpelevich Finiteness]\label{maintheorem2} 
For $\lambda^{\vee}, \mu^{\vee}\in \Lambda^{\vee}$, the set $K\backslash K\pi^{\mu^{\vee}}U^{+}\cap K\pi^{\lambda^{\vee}}U^{-}$ is finite. Moreover, it is empty unless $ \mu^{\vee}\le \lambda^{\vee}$.
\end{theorem}
\noindent For $\mathcal{S}_{\lambda^{\vee}}$, one needs to prove 
\begin{theorem}[Spherical Finiteness]\label{maintheorem3} 
For $\lambda^{\vee}, \mu^{\vee}\in \Lambda^{\vee}$ with $\lambda^{\vee}$ dominant, the coset space $K\backslash K\pi^{\mu^{\vee}}U^{+}\cap K\pi^{\lambda^{\vee}}K$ is finite. Moreover, it is empty unless $ \mu^{\vee}\le \lambda^{\vee}$.
\end{theorem}
 For (untwisted) {\it affine} Kac-Moody groups, Braverman et. al. obtain Theorems~\ref{maintheorem1}, \ref{maintheorem2} and \ref{maintheorem3}. Note that their proofs for the second part of the Spherical and Gindikin-Karpelevich Finiteness (the inequalities $ \mu^{\vee}\le \lambda^{\vee}$) extend to general Kac-Moody groups without any change. So, we do not prove this part here in this paper.
 
  The first part of the (untwisted) {\it affine} Gindikin-Karplevich Finiteness was proven for $\lambda^{\vee}=0$ by showing that:
 \begin{enumerate}
\item [(a)] $K\pi^{\mu^{\vee}}U^{+}\cap KU^{-}=\cup_{w\in W} K\pi^{\mu^{\vee}}U^{+}\cap K\mathcal{V}^{-}_{w}$, where for each $w\in W$, $\mathcal{V}^{-}_{w}$ is a certain subset of $U^{-}$ defined in  Section 3 of \cite{BGKP}.
\item[(b)] A corollary of the Kac-Moody generalization \cite[Lemma~18.2]{BFG} of a representation theoretic construction due to A. Joseph \cite{Jos2, Jos} implies that there are finitely many $w$ which appear in the above union.
\item[(c)] By using the completions, it is then proved that for each such $w$, $K\backslash K\pi^{\mu^{\vee}}U^{+}\cap K\mathcal{V}^{-}_{w}$ is finite.
\end{enumerate}
\noindent Next, the Gindikin-Karpelevich finiteness is used to get the Approximation Theorem as well as Spherical Finiteness. Finally, by combining these results with an affine generalization of the Macdonald's formula for $\mathcal{S}_{\lambda^{\vee}}$ from \cite{BKP}, the following affine version of the Gindikin-Karpelvich formula is obtained
\begin{eqnarray}
\mathscr{G}_{0}=\frac{1}{\mathfrak{m}} \prod_{\alpha^{\vee}\in \Delta^{\vee}_{+}}\left (\frac{1-q^{-1}e^{-\alpha^{\vee}}}{1-e^{-\alpha^{\vee}}}\right)^{m(\alpha^{\vee})},
\end{eqnarray}
where $m(\alpha^{\vee})$ is the multiplicity of the coroot $\alpha^{\vee}$ and $\mathfrak{m}$ is a $W$-invariant factor which depends on the Langlands-dual root system of given affine Lie algebra. Recently, for general Kac-Moody settings the factor $\mathfrak{m}$ was studied and various properties were listed by D. Muthiah, A. Pusk\'as and I. Whitehead in \cite{MPW}.

\subsection{}\label{sub5} As stated in previous subsection, our key objective in this paper is to obtain the proofs of Theorems~\ref{maintheorem1}, \ref{maintheorem2} and \ref{maintheorem3} for general Kac-Moody groups over non-archimedean local fields. Though Braverman-Garland-Kazhdan-Patnaik's work \cite{BGKP} is the main motivation of this paper, our approach to prove these results is a little different. Unlike the affine case, we prove Theorem~\ref{maintheorem1} independently of the finiteness result Theorem~\ref{maintheorem2}. Theorem~\ref{maintheorem1} has also been proven by A. H\'{e}bert in \cite[Theorem 6.1]{Aug}, but our proof is perhaps more elementary and can also be used to obtain the Iwahori version of the assertion (see Proposition~\ref{Prop1}). 

Next, we turn to Theorem~\ref{maintheorem2}. For the first part of the assertion, our method of proof restricts us to put a further condition on $\lambda^{\vee}$. Namely, we first prove
\begin{theorem}[Weak Spherical Finiteness]\label{maintheorem4} 
Let $\mu^{\vee}\in \Lambda^{\vee} $. For $\lambda^{\vee}\in \Lambda^{\vee}_{+}$ regular and sufficiently dominant, the set $K\backslash K\pi^{\mu^{\vee}}U^{+}\cap K\pi^{\lambda^{\vee}}K$ is finite.
\end{theorem} 
This theorem is proven by getting the finiteness at the Iwahori level. The Iwahori level questions are indexed by the Weyl group $W$. So, first we show that there are finitely many elements of the Weyl group which contribute (Section~\ref{section4}), each indexing an Iwahori piece of our sum.  In Section~\ref{Finitesupport}, we introduce a certain integral and show that it satisfies a recursion relation in terms of certain Demazure-Lusztig operators. Our objective in so doing is to obtain the finiteness of the Iwahori piece. In Section~\ref{sectionfiberfin},  we establish a relation between an Iwahori piece and a level set of the integral which allows us to complete the proof of the Weak Spherical Finiteness. 

In Section~\ref{secapplication}, we discuss the applications of the results proven in the previous subsection. First, we apply the Approximation Theorem and the Weak Spherical Finiteness to show that if $\mu^{\vee}\in Q^{\vee}_{-}$ is very small as compared to $\lambda^{\vee}=0$, then $K\backslash K\pi^{\mu^{\vee}}U^{+}\cap KU^{-}$ is finite and this together with Lemma~\ref{GKrem} imply the Gindikin-Karplevich Finiteness. Then, as in the affine case, we use the Gindikin-Karplevich Finiteness to get the proof of the Spherical Finiteness.

\subsection{}\label{sub7}  These finiteness theorems for general Kac-Moody settings have also appeared in some other publications. H\'{e}bert has obtained the proof of Theorem~\ref{maintheorem3} for general Kac-Moody settings. Theorem~\ref{maintheorem2} (for $\lambda=0$) is shown to be true by S. Gaussent and G. Rousseau in \cite{GR2}. Both of these proofs involve the techniques based on the use of geometric objects known as {\it masures}, introduced by Gaussent and Rousseau in \cite{GR}. These are an analogue of the Bruhat-Tits {\it buildings} for groups over local fields. On the other hand, our strategy for proving the assertions of these theorems is elementary, algebraic in nature and relies on the use of the representation theory. It would be interesting to compare these two techniques in more detail.

\subsection*{Acknowledgement} We would like to thank Manish M. Patnaik for his support, several discussions and valuable feedbacks which made this work complete.

\section{Kac-Moody Algebras}\label{Bassettings}
\subsection{Generalized Cartan Matrices}\label{GCM}
Let $I$ be a finite set of cardinality $l$ and  $A=(a_{ij})_{i,j\in I}$ be a square matrix such that for all $i,j\in I$,
\begin{enumerate}
\item[(i)] $a_{ii}=2$; for $i\ne j$, $a_{ij}$ are non positive integers; and $a_{ij}=0$ if and only if $a_{ji}=0$, 
\item[(ii)] there exist a diagonal matrix $D$ and a positive definite matrix $P$ such that $A=DP$. 
\end{enumerate}
 A matrix $A$ is said to be a {\it Cartan matrix} if it satisfies (i) \& (ii) and if $A$ satisfies only (i), then it is called a {\it generalized Cartan matrix} (GCM). 

\noindent A GCM $A$ is said to be equivalent to another matrix $B$, if $B$ is obtained by reordering the indices $i$ and $j$, and vice versa. If a GCM $A$ is equivalent to a block diagonal matrix with more than one block, we say it is a {\it decomposable} GCM; otherwise, it is known as an {\it indecomposable} GCM. 
An indecomposable GCM $A$ can be classified into the following three types,
\begin{enumerate}
\item[(a)]  {\bf Finite:} If $A$ is positive-definite. In this case the determinant of $A$ is positive.
\item[(b)] {\bf Affine:} If $A$ is positive-semidefinite. In this case $det(A)=0$.
\item[(c)] {\bf Indefinite:} If $A$  is neither finite nor affine type.
\end{enumerate}
As in \cite{Kac1}, suppose $( \mathfrak{h}, \Pi, \Pi^{\vee})$ is a realization of GCM $A$, that is, $\mathfrak{h}$ is a complex vector space of finite dimension; $\Pi^{\vee}=\{\alpha_{i}^{\vee}\}_{i\in I}\subset \mathfrak{h}$, $\Pi=\{\alpha_{i}\}_{i\in I}\subset \mathfrak{h}^{\ast}$ are two linearly independent sets, such that for all $i,j\in I$, $\alpha_{j}(\alpha_{i}^{\vee})=a_{ij}$; and, $dim(\mathfrak{h})=2l-rank(A)$. The elements of $\Pi$ are called  {\it simple roots} and those of $\Pi^{\vee}$ are known as {\it simple coroots}. A realization $( \mathfrak{h}, \Pi, \Pi^{\vee})$ is said to be a {\it decomposable realization} if $\mathfrak{h}=\mathfrak{h}_{1}\oplus \mathfrak{h}_{2}$, $\Pi^{\vee}=(\Pi_{1}^{\vee}\times \{0\})\cup (\{0\}\times \Pi_{2}^{\vee})$ and  $\Pi=(\Pi_{1}\times \{0\})\cup (\{0\}\times \Pi_{2})$, where $( \mathfrak{h}_{1}, \Pi_{1}, \Pi^{\vee}_{1})$ and $( \mathfrak{h}_{2}, \Pi_{2}, \Pi^{\vee}_{2})$ are realizations themselves. A realization is called {\it indecomposable realization} if it is not a decomposable realization. In \cite[Proposition 1.1]{Kac1}, Kac asserts that an indecomposable GCM $A$ corresponds to an indecomposable realization $( \mathfrak{h}, \Pi, \Pi^{\vee})$ which is unique up to equivalence in the following sense: A realization $( \mathfrak{h}, \Pi, \Pi^{\vee})$ is said to be  equivalent to another realization $( \mathfrak{h}', \Pi', {\Pi'}^{\vee})$, if there exists an isomorphism $\phi\in \operatorname{Hom}_{\mathbb{C}}(\mathfrak{h}, \mathfrak{h}')$, such that $\phi(\Pi^{\vee})={\Pi'}^{\vee}$ and $\phi^{\ast}(\Pi)=\Pi'$, where $\phi^{\ast}$ is the induced isomorphism of the dual spaces of $\mathfrak{h}$ and $\mathfrak{h}'$. Furthermore, this isomorphism is unique if $det(A)\ne 0$.

\subsection{Presentation of Kac-Moody Algebras}\label{prekmal}

A GCM of finite type, up to equivalence, corresponds to a unique finite dimensional complex semisimple Lie algebra $\mathfrak{g}$ up to isomorphism. In \cite{Ch1}, Chevalley showed that a finite dimensional complex semisimple Lie algebra admits an abstract presentation as a set of generators satisfying relations in terms of entries of a Cartan matrix. Later, Serre in \cite{S1} proved that these relations are the defining relations of $\mathfrak{g}$. Insprired by this, Victor G. Kac \cite{Kac0} and R. Moody \cite{Mo0} independently introduced a new class of Lie algebras by giving a finite presentation in terms of entries of a GCM $A$. These Lie algebras are infinite dimensional generalizations of finite dimensional complex semisimple Lie algebras. 

To describe these Lie algebras, let $A=(a_{ij})_{i,j\in I}$ be a GCM and  $(\mathfrak{h}, \Pi, \Pi^{\vee})$ the associated realization. We define $\mathfrak{g}$ as the Lie algebra generated by $\mathfrak{h}$ and the $2n$ generators $\{e_{i}, f_{i}\}_{i\in I}$ subject to the following relations:
\begin{enumerate}
\item[(1)] $[\mathfrak{h}, \mathfrak{h}]=0.$
For all $i,j\in I,$
\item[(2)] $[e_{i},h ]=\alpha_{i}(h)e_{i},\;h\in \mathfrak{h}$,
\item[(3)] $[f_{i},h ]=-\alpha_{i}(h)f_{i},\;h\in \mathfrak{h}$,
\item[(4)] $[e_{i},f_{j} ]=\delta_{ij}\alpha_{i}^{\vee}$, where $\delta_{ij}$ is the {\it Kronecker delta },
\item[(5)] For $i\ne j$,  ${\operatorname{ad}_{e_{i}}}^{(-a_{ij}+1)}(e_{j})=0$ and ${\operatorname{ad}_{f_{i}}}^{(-a_{ij}+1)}(f_{j})=0$, where $\operatorname{ad}$ is the adjoint representation of $\mathfrak{g}$.
\end{enumerate}

The Lie algebra $\mathfrak{g}$ is known as a Kac-Moody algebra. The subalgebra $\mathfrak{h}$ is called {\it Cartan subalgebra} and $\{e_{i}, f_{i}\}_{i\in I}$ are known as the {\it Chevalley generators} of $\mathfrak{g}$. Let $\mathfrak{n}^{+}$ and $\mathfrak{n}^{-}$ be the subalgebras generated by  $\{e_{i}\}_{i\in I}$ and $\{ f_{i}\}_{i\in I}$, respectively. Then $\mathfrak{g}$ has a vector subspace decomposition
\begin{eqnarray}\label{Tridec}
\mathfrak{g}&=&\mathfrak{n}^{-}\oplus \mathfrak{h}\oplus \mathfrak{n}^{+}.
\end{eqnarray}
This decomposition is known as the {\it triangular decomposition} of $\mathfrak{g}$.

A Kac-Moody algebra $\mathfrak{g}$ falls in one of the three categories finite, affine or indefinite, according to the type of GCM $A$ (cf. Subsection~\ref{GCM}). A finite type Kac-Moody algebra is a complex semisimple Lie algebra. There are four infinite families of complex semisimple Lie algebras $A_{n},\; n\ge1$; $B_{n},\; n\ge2$; $C_{n},\; n\ge3$; $D_{n},\; n\ge 4$. These four families are known as {\it classical Lie algebras}. In addition to the classical Lie algebras there are five so-called {\it exceptional Lie algebras} Lie algebras, $G_{2}$, $F_{4}$, $E_{6}$, $E_{7}$ and $E_{8}$. 
An affine Kac-Moody algebra is an extensions of loop algebra of finite dimensional semisimple Lie algebra. So, once the classification in finite dimension is known, the classification of affine Kac-Moody algebras becomes possible. These Lie algebras are described with the symbol $X^{r}$ with $X=A_{n}, B_{n}, C_{n}, D_{n}, E_{6}, E_{7}, E_{8}, F_{4}$ and $G_{2}$, and $r=1,2, 3$. When $r=1$, the corresponding Lie algebras are called the {\it untwisted affine Lie algebras}. When $r=2, 3$, the corresponding Lie algebras are called the {\it twisted affine Lie algebras}. 
The Kac-Moody algebras of indefinite type are the least understood and the classification of the corresponding root systems has not yet been achieved. However, a subclass known as the {\it hyperbolic Kac-Moody algebras} is well known. These correspond to the GCMs $A$ such that every proper, indecomposable principal submatrix of $A$ is either of finite or affine type. In this case, we have $det(A)<0$. These Lie algebras and the related data have been studied in the literature, for example, see \cite{CCCMNNP, CFL, CN, Li, Zha}.

\subsection{Roots and Weyl Group}\label{rootsandweyl}
Let $Q=\oplus_{i\in I}\mathbb{Z}\alpha_{i}$ and $Q^{\vee}=\oplus_{i\in I}\mathbb{Z}\alpha_{i}^{\vee}$ be the root and coroot lattice, respectively. The Lie algebra $\mathfrak{g}$ admits another decomposition $\mathfrak{g}=\oplus_{\alpha\in Q}\mathfrak{g}_{\alpha}$ under the adjoint action of $\mathfrak{h}$ where 
$$\mathfrak{g}_{\alpha}=\{x\in \mathfrak{g}\mid [h,x]=\alpha(h)x; \forall \;h\in \mathfrak{h}\}.$$ 
This is known as the {\it root space decomposition}. If $\alpha\ne 0$ and $\mathfrak{g}_{\alpha}\ne 0$, the subspace $\mathfrak{g}_{\alpha}$ is known as the {\it root space} of $\alpha$; its dimension $m({\alpha})$ is called the multiplicity of $\alpha$; and, $\alpha$ is called a {\it root}. We denote the set of roots by $\Delta$. Every root can be written as an integral combination of the simple roots, with the coefficients, either all positive or all negative integers; a root is called positive or negative, accordingly. We denote the set of positive roots by $\Delta_{+}$, the set of negative roots by $\Delta_{-}$ one has a disjoint union,
$$\Delta=\Delta_{+}\sqcup \Delta_{-}.$$

In the study of finite dimensional complex semisimple Lie algebras, a symmetric, nondegenerate and invariant bilinear form known as the {\it Killing form} plays a significant role. An analogue of the Killing form in the general setting can be defined if the GCM $A$ is {\it symmetrizable}. That is, $A=DB$, where $D$ is a non-singular diagonal matrix and $B$ is a symmetric matrix. Restriction of this bilinear form to $\mathfrak{h}$ is also nondegenerate; therefore it induces an isomorphism between $\mathfrak{h}$ and $\mathfrak{h}^{\ast}$. The natural pairing between $\mathfrak{h}$ and $\mathfrak{h}^{\ast}$ is denoted by
$$\langle -,-\rangle:\mathfrak{h}^{\ast}\times \mathfrak{h}\longrightarrow \mathbb{C}.$$ 
Let $Q_{+}=\oplus_{i\in I}\mathbb{Z}_{\ge 0}\alpha_{i}$. The space $\mathfrak{h}^{\ast}$ can be equipped with a partial order $\le $ defined as: $\mu\le \lambda$ if and only if $\lambda-\mu\in Q_{+}$, for all $\lambda, \mu\in \mathfrak{h}^{\ast}$. Similarly, we can define a partial order on $\mathfrak{h}$, which we denote by the same symbol $\le $, by setting $Q^{\vee}_{+}=\oplus_{i\in I}\mathbb{Z}_{\ge 0}\alpha_{i}^{\vee}$ and imposing the same defining condition as above. An element $\lambda\in \mathfrak{h}^{\ast}$ is {\it integral} if $\langle \lambda,\alpha_{i}^{\vee}\rangle\in \mathbb{Z}$, is {\it dominant} if $\langle \lambda,\alpha_{i}^{\vee}\rangle \ge 0$, and is called {\it regular} if $\langle \lambda,\alpha_{i}^{\vee}\rangle \ne 0$, for all $i\in I$. Let 
\begin{eqnarray}
\Lambda:=\{\lambda\in \mathfrak{h}^{\ast}\mid \lambda(h)\in \mathbb{Z},\;\forall h\in \mathfrak{h} \}
\end{eqnarray}
 and $\Lambda^{\vee}=\operatorname{Hom}_{\mathbb{Z}}(\Lambda, \mathbb{Z}) $ be the weight and coweight lattice, respectively. We denote by $\Lambda_{+}$ the set of dominant weights and $\Lambda_{reg}$ the set of regular weights. Similarly we define the sets $\Lambda_{+}^{\vee}$ and $\Lambda_{reg}^{\vee}$.
For $i\in I$, let us define a map $w_{i}=w_{\alpha_{i}}$ on $\mathfrak{h}^{\ast}$ by setting
\begin{eqnarray} 
w_{i}(\lambda)=\lambda-\lambda(\alpha_{i}^{\vee})\alpha_{i}=\lambda-\langle \lambda, \alpha_{i}^{\vee}\rangle \alpha_{i},
\end{eqnarray}
for all $\lambda\in \mathfrak{h}^{\ast}.$ This map is a reflection in the hyperplane 
\begin{eqnarray}
(\alpha_{i}^{\vee})^{\perp}=\{\mu\in \mathfrak{h}\mid \langle \mu, \alpha_{i}^{\vee}\rangle=0\}.
\end{eqnarray}
 The reflection $w_{i}$ is called a {\it simple root reflection} or the {\it Weyl reflection} and the group $W\subset \Aut(\mathfrak{h}^{\ast})$ generated by the simple root reflections $w_{i}$ for all $i\in I$, is called the {\it Weyl group}. The Weyl group also acts on $\mathfrak{h}$ via the following formula, 
\begin{eqnarray}
w_{i}(\lambda^{\vee})=\lambda^{\vee}-\langle \alpha_{i}, \lambda^{\vee}\rangle \alpha_{i}^{\vee},
\end{eqnarray}
for all $\lambda^{\vee}\in \mathfrak{h}.$
If the Kac-Moody algebra $\mathfrak{g}$ is of affine or indefinite type, the set $\Delta$ of roots admits another partition 
\begin{eqnarray}
\Delta=\Delta^{re}\cup\Delta^{im},
\end{eqnarray}
 where 
$\Delta^{re}=W\Pi$ and $\Delta^{im}=\Delta-\Delta^{re}$. The elements of $\Delta^{re}$ are called {\it real roots} and those of $\Delta^{im}$ are known as the {\it imaginary roots}. The transpose of GCM $A$ corresponds to the realization $(\mathfrak{h}^{\ast}, \Pi^{\vee}, \Pi)$ and a dual root system $\Delta^{\vee}\subset \mathfrak{h} $, which is called the set of coroots. This set also admits the disjoint decompositions 
\begin{eqnarray}
\Delta^{\vee}=\Delta^{\vee}_{+}\cup \Delta^{\vee}_{ -},\;\;\; \Delta^{\vee}=\Delta^{\vee, re}\cup \Delta^{\vee, im}
\end{eqnarray}
 and there is a bijection $\alpha\mapsto \alpha^{\vee}$ between $\Delta$ and $\Delta^{\vee}$.
\subsection{Highest Weight Representation}
Let $\mathfrak{g}$ be a Kac-Moody algebra, $\mathcal{U}=\mathcal{U}(\mathfrak{g})$ be the universal enveloping algebra of $\mathfrak{g}$. The triangular decomposition (\ref{Tridec}) of $\mathfrak{g}$ yields the triangular decomposition 
$$\mathcal{U}=\mathcal{U}(\mathfrak{n}^{+})\oplus \mathcal{U}( \mathfrak{h})\oplus \mathcal{U}( \mathfrak{n}^{-})$$ 
of $\mathcal{U}$.
For $\lambda\in \Lambda_{+}$, a $\mathfrak{g}$ representation $V=V^{\lambda}$ over $\mathbb{C}$ is a {\it highest weight representation} with the highest weight ${\lambda}\in \mathfrak{h}^{*}$ and a highest weight vector $v_{\lambda}$ if:
\begin{enumerate}
\item[(i)] $\mathfrak{n}^{+}v_{\lambda}=0$,
\item[(ii)] $h.v_{\lambda}=\lambda(h)v$ for all $h\in \mathfrak{h}$, 
 \item[(iii)] $V=\mathcal{U}v_{\lambda}$.
 \end{enumerate}
 Moreover, if 
\begin{enumerate}
\item[(iv)] for all $i\in I$, $e_{i}$ and $f_{i}$ act as locally nilpotent operators on $V$, that is, for each $v\in V$ there exist integers $M$ and $N$ such that  $e_{i}^{M}v=f_{i}^{N}v=0$, \\ then the space $V$ is said to be an {\it integrable} highest weight representation.
 \end{enumerate}

The space $V$ has a {\it weight space decomposition} 
\begin{eqnarray}
V=\oplus_{\mu\in \mathfrak{h}^{\ast}}V_{\mu},
\end{eqnarray} 
where $V_{\mu}=\{v\in V\mid hv=\mu(h)v,\; \forall\; h\in \mathfrak{h}\}.$ 
 Let us denote by $P_{\lambda}$ the set of weights of $V$. For $\mu\in P_{\lambda}$ 
 \begin{eqnarray}\label{projmap4}
 \eta_{\mu}:V\longrightarrow V_{\mu}
 \end{eqnarray}
 denotes the projection map. Unless otherwise specified, throughout this paper our highest weight representation shall be integrable. 
 
 The set of weights $P_{\lambda}$ inherits the partial order from $\mathfrak{h}^{\ast}$ and each $\mu\in P_{\lambda}$ satisfies $\mu\le \lambda$ which implies $ \lambda-\mu=\sum_{i\in I}n_{i}\alpha_{i}$ with $n_{i}\in \mathbb{Z}_{\ge 0}$ for all $i\in I$. For $\mu\in P_{\lambda}$, we define the {\it depth} of $\mu$ as $depth(\mu)=\sum_{i\in I}n_{i}$.  
Let $V_{\mathcal{O}}$ the integral lattice in $V$. For $v\in V$, set $Ord(v)=min_{n\in\mathbb{Z}}\pi^{n}v\in V_{\mathcal{O}}$ and define a norm on $V$ as
\begin{eqnarray}\label{repthnorm4}
||v||:=q^{Ord(v)},
\end{eqnarray}
for all $v\in V$. An element $v\in V$ is said to be  {\it primitive} if $||v||=1$; we will always choose the highest weight vector $v_{\lambda}$ to be a primitive element. 
\noindent We end this section with the following lemma which can be verified easily.
\begin{lemma}\label{sec1l1}
If $v,w\in V$ belong to different weight spaces then $||v+w||\ge ||v||$.
\end{lemma}

\section{Kac-Moody Groups}
Let $\mathfrak{g}$ be complex Kac-Moody algebra of finite type, that is, $\mathfrak{g}$ is a finite dimensional complex semisimple Lie algebra. By using the $\mathbb{Z}$-form of $\mathfrak{g}$, Chevalley in \cite{Che1} found a uniform procedure to associate an algebraic group $G$ with $\mathfrak{g}$ over an arbitrary field $\mathbb{K}$, which is semisimple if $\mathbb{K}$ is algebraically closed. These groups are known as the {\it Chevalley groups}. A detailed exposition and complete construction of the Chevalley groups can be found in R. Steinberg's lecture notes \cite{St}. 

In the general setting, the association of a group with a Kac-Moody algebra corresponding to a GCM of affine or indefinite type is a complex problem. There are three constructions of Kac-Moody groups which can be found in the literature. For this paper, we choose the Kac-Moody groups constructed by  L. Carbone and H. Garland  but our results are independent of this choice.

\subsection{Carbone-Garland Construction}\label{CGKMg}
 Steinberg's construction of Chevalley groups is a natural candidate for an infinite dimensional generalization. Carbone and Garland extended this construction to Kac-Moody root systems over arbitrary fields in \cite{CG}. Recently this construction has been generalized to define Kac-Moody groups over  $\mathbb{Z}$ and arbitrary rings by Carbone et al. in \cite{AC, CW}. 
 
 \subsubsection{ $\mathbb{Z}$-Forms: Pathway to Arbitrary Fields}
 We retain the notation of Kac-Moody algebra $\mathfrak{g}$ and the related data from Section~\ref{Bassettings}. Let $\mathcal{U}$, $\mathcal{U}(\mathfrak{n}^{+})$ and $\mathcal{U} (\mathfrak{n}^{-})$ be the universal enveloping algebras of $\mathfrak{g}$, $\mathfrak{n}^{+}$ and $\mathfrak{n}^{-}$, respectively. Let $S(\mathfrak{h})$ be the symmetric algebra of $\mathfrak{h}$,  Tits in \cite{Ti10} asserts that the canonical map 
 \begin{eqnarray}
 \mathcal{U}(\mathfrak{n}^{+})\otimes S(\mathfrak{h})\otimes \mathcal{U}(\mathfrak{n}^{-})\longrightarrow \mathcal{U}
 \end{eqnarray}
 is a bijection. Next, we introduce some notions on an associative algebra $\mathcal{A}$ over $\mathbb{C}$ which will be used later.
 \begin{definition}
 A $\mathbb{Z}$ form of $\mathcal{A}$ is a $\mathbb{Z}$ subalgebra $\mathcal{A}_{\mathbb{Z}}$ of $\mathcal{A}$ such that the canonical map $\mathcal{A}_{\mathbb{Z}}\otimes \mathbb{C}\longrightarrow \mathcal{A}$ is a bijection.
 \end{definition}
 \begin{definition}
 For $a\in \mathcal{A}$ and $n\in \mathbb{Z}_{\ge 0}$, we define the following elements of $\mathcal{A}$,
 \begin{eqnarray}
 a^{(n)}&:=&\frac{a^{n}}{n!}\\
 \binom{a}{n}&:=&\frac {a(a-1)(a-2)\dots (a-n+1)}{n!}.
 \end{eqnarray}
 \end{definition}
 
 Let us denote by $\mathfrak{t}$ and $\mathfrak{t}^{\vee}$ the linear span of $\alpha_{i}$ and $\alpha_{i}^{\vee}$ for $i\in I$, respectively. For $i\in I$ and $n\in\mathbb{Z}_{\ge 0}$, $\mathcal{U}_{i,+}$ (resp. $\mathcal{U}_{i,-}$) be the subring $\sum_{n} \mathbb{Z} e_{i}^{(n)}$ (resp. $\sum_{n}\mathbb{Z} f_{i}^{(n)}$) of $\mathcal{U}$. Let $\mathcal{U}_{\mathbb{Z},+}$ (resp. $\mathcal{U}_{\mathbb{Z},-}$) be a subring of $\mathcal{U}(\mathfrak{n}^{+})$ (resp. of $\mathcal{U}(\mathfrak{n}^{-})$) generated by $\mathcal{U}_{i,+}$ (resp. $\mathcal{U}_{i,-}$) for all $i\in I$. Then $\mathcal{U}_{\mathbb{Z},+}$ and $\mathcal{U}_{\mathbb{Z},-}$ are the $\mathbb{Z}$-subalgebras of $\mathcal{U}(\mathfrak{n}^{+})$ and  $\mathcal{U}(\mathfrak{n}^{-})$, respectively \cite[p.556]{Ti9}. Let $\mathcal{U}_{\mathbb{Z},0}$ be the $\mathbb{Z}$-subalgebra of the universal enveloping algebra $S(\mathfrak{h})$ generated by $\binom {\lambda}{n}$ ($\lambda\in \mathfrak{t}^{\vee}$) and $\mathcal{U}_{\mathbb{Z}}$ be  the $\mathbb{Z}$-subalgebra of the universal enveloping algebra $\mathcal{U}$ generated by  $\mathcal{U}_{i,+}$, $\mathcal{U}_{i,-}$ and $\binom {\lambda}{n}$ for $i\in I$, $n\in \mathbb{Z}_{\ge 0}$ and $\lambda\in \mathfrak{t}^{\vee}$.  
 
 We state the following result from \cite[p.106]{MT} and \cite{Ti11} without giving its proof.
\begin{proposition}\label{propdec} We have the following
\begin{enumerate}[before=\itshape,font=\normalfont]
\item[(i)] $\mathcal{U}_{\mathbb{Z},+}$, $\mathcal{U}_{\mathbb{Z},-}$ and $\mathcal{U}_{\mathbb{Z},0}$ are the $\mathbb{Z}$-forms of $\mathcal{U}(\mathfrak{n}^{+})$, $\mathcal{U}(\mathfrak{n}^{-})$ and $S(\mathfrak{h})$, respectively.
\item[(ii)] $\mathcal{U}_{\mathbb{Z}}$ is the $\mathbb{Z}$-form of $\mathcal{U}$.
\item[(iii)] The product map 
\begin{eqnarray}
\mathcal{U}_{\mathbb{Z},-}\otimes\mathcal{U}_{\mathbb{Z},0}\otimes \mathcal{U}_{\mathbb{Z},+}\longrightarrow \mathcal{U}_{\mathbb{Z}} 
\end{eqnarray}
 is an isomorphism of $\mathbb{Z}$-modules.
\end{enumerate}
\end{proposition}
The $\mathbb{Z}$-forms of $\mathcal{U}$ and its subalgebras will be used to define the Kac-Moody algebra $\mathfrak{g}_{\mathbb{K}}$ over an arbitrary field $\mathbb{K}$. Let $\mathfrak{g}_{\mathbb{Z}}=\mathfrak{g}\cap \mathcal{U}_{\mathbb{Z}}$, $\mathfrak{n}^{\pm}_{\mathbb{Z}}=\mathfrak{n}^{\pm}\cap \mathcal{U}_{\mathbb{Z}}$ and $\mathfrak{t}^{\vee}_{\mathbb{Z}}=\mathfrak{t}^{\vee}\cap \mathcal{U}_{\mathbb{Z},0}$. The following proposition on page 78 of \cite{MT} implies that $\mathfrak{g}_{\mathbb{Z}}$ is $\mathbb{Z}$-form of $\mathfrak{g}$.

\begin{proposition} The sum map\\
 $$\mathfrak{n}^{+}_{\mathbb{Z}}\oplus \mathfrak{t}^{\vee}_{\mathbb{Z}}\oplus \mathfrak{n}^{-}_{\mathbb{Z}}\longrightarrow \mathfrak{g}_{\mathbb{Z}} $$
is a bijection.
 \end{proposition}
\noindent For a field $\mathbb{K}$, set
\begin{eqnarray*}
\mathfrak{n}^{\pm}_{\mathbb{K}}&:=&\mathfrak{n}^{\pm}_{\mathbb{Z}}\otimes\mathbb{K},\; \mathfrak{t}^{\vee}_{\mathbb{K}}:=\mathfrak{t}^{\vee}_{\mathbb{Z}}\otimes \mathbb{K},\;\mathfrak{g}_{\mathbb{K}}:=\mathfrak{g}_{\mathbb{Z}}\otimes\mathbb{K},\\
\mathcal{U}_{\mathbb{K},\pm}&:=&\mathcal{U}_{\mathbb{Z},\pm}\otimes \mathbb{K},\; \mathcal{U}_{\mathbb{K},0}:=\mathcal{U}_{\mathbb{Z},0}\otimes \mathbb{K},\;\mathcal{U}_{\mathbb{K}}:=\mathcal{U}_{\mathbb{Z}}\otimes \mathbb{K}.
 \end{eqnarray*}
Then $\mathfrak{g}_{\mathbb{K}}$ is a Kac-Moody algebra over $\mathbb{K}$ and it admits the root spaces decomposition
\begin{eqnarray}
\mathfrak{g}_{\mathbb{K}}=\mathfrak{t}^{\vee}_{\mathbb{K}}\oplus (\oplus_{\alpha\in \Delta} \mathfrak{g}_{\alpha,\mathbb{K}}),
\end{eqnarray}
where for each $\alpha\in \Delta$, $\mathfrak{g}_{\alpha,\mathbb{K}}=(\mathfrak{g}_{\alpha}\cap \mathfrak{g}_{\mathbb{Z}})\otimes  \mathbb{K}$.

\subsubsection{Minimal Kac-Moody Group}\label{GCgroup}
As in Steinberg's presentation of Chevalley groups, the second essential ingredient in Carbone and Garland's construction of Kac-Moody groups is an integrable representation having a stable lattice. We begin with a description of this lattice.

Let $V=V^{\lambda}$ be an integrable highest weight representation with the highest weight ${\lambda}$ and the highest weight vector $v_{\lambda}$. As in the finite dimensional case \cite{St}, a $\mathbb{Z}$-lattice $V_{\mathbb{Z}}$ is constructed by setting $V_{\mathbb{Z}}=\mathcal{U}_{\mathbb{Z}}v_{\lambda}$. The following lemma gives a more concrete description of the lattice $V_{\mathbb{Z}}$.
\begin{lemma}
We have
\begin{eqnarray}
V_{\mathbb{Z}}=\mathcal{U}_{\mathbb{Z},-}(v_{\lambda}).
\end{eqnarray}
\end{lemma}

\begin{corollary}
The space $V_{\mathbb{Z}}$ is a $\mathbb{Z}$-form and an {\it admissible lattice} of $V$, that is, for $i\in I$ and for some $n\ge 0$
 \begin{eqnarray*}
 e_{i}^{(n)}V_{\mathbb{Z}}\subset V_{\mathbb{Z}}; \; \; f_{i}^{(n)}V_{\mathbb{Z}}\subset V_{\mathbb{Z}}.
  \end{eqnarray*}
\end{corollary}

For each weight $\mu$ of $V$ and the corresponding weight space $V_{\mu}$, we set $V_{\mu, \mathbb{Z}}=V_{\mu}\cap V_{\mathbb{Z}}$. Then $V_{\mathbb{Z}}=\oplus_{\mu\in P_{\lambda}}V_{\mathbb{Z},\mu}.$  For a field $\mathbb{K}$, let $V_{\mathbb{K}}:=V_{\mathbb{Z}}\otimes \mathbb{K},\;\;V_{\mathbb{K},\mu}:=V_{\mathbb{Z},\mu}\otimes \mathbb{K}$
and for $t\in \mathbb{K}$
 \begin{eqnarray}
 \chi_{\alpha_{i}}(t):=\sum_{n\ge 0}t^{n}e_{i}^{(n)}=e^{te_{i}},\;\; \chi_{-\alpha_{i}}(t) :=\sum_{n\ge 0}t^{n}f_{i}^{(n)}=e^{tf_{i}}.
\end{eqnarray} 
Since $e_{i}$ and $f_{i}$ act as locally nilpotent operators, $ \chi_{\pm\alpha_{i}}(t)$ are well defined automorphisms of $V_{\mathbb{K}}$. The minimal Kac-Moody group $G(\mathbb{K})$ of Carbone and Garland is the subgroup of $\Aut(V_{\mathbb{K}})$ generated by the elements  $\chi_{\pm \alpha_{i}}(t)$ with $ t\in\mathbb{K}$ and $i\in I$ \cite[Section 5]{CG}. Similar to the Chevalley groups \cite[Lemma 27]{St}, the group constructed above depends on the integrable highest weight representation $V$ and the choice of an admissible lattice in $V$. 

 \subsubsection{Tits Axioms and $BN$-Pairs}
 In what follows, we shall consider Carbone-Garland's Kac-Moody group $G(\mathbb{K})$ over a field $\mathbb{K}$ and denote it by $G$, by dropping $\mathbb{K}$ from our notation.
 For $i\in I$ and $t\in \mathbb{K}^{\ast}$, set 
  \begin{eqnarray}
w_{i}(t)&:=&\chi_{\alpha}(t)\chi_{-\alpha}(-t^{-1})\chi_{\alpha}(t),\;\;\;   \tilde{w}_{i}:=w_{i}(1)\label{relationtorus1}\\
h_{i}(t)&:=& \tilde{w}_{i}(t) \tilde{w}_{i}^{-1}.\label{relationtorus2}
 \end{eqnarray}
Let $H$ be the subgroup generated by the elements $h_{i}(t)$ for all $i\in I$ and $t\in \mathbb{K}^{\ast}$. Let $\alpha\in \Delta^{re}$ then $\alpha=w\alpha_{i}$ form some $w\in W$ and simple root $\alpha_{i}$, $i\in I$. For $t\in \mathbb{K}$, we set 
\begin {eqnarray}
\chi_{\alpha}(t)=w\chi_{\alpha_{i}}(t)w^{-1}
\end{eqnarray} 
One can check that for $t\in\mathbb{K}$, we have $\chi_{\alpha}(t)\in \Aut(V_{\mathbb{K}})$.
Associated with $\alpha\in \Delta^{re}$, a {\it root group} is defined as,
 $$U_{\alpha}=\{\chi_{\alpha}(t)\mid  t\in \mathbb{K}\}.$$
 Continuing with $\alpha\in \Delta^{re}_{\pm}$, let $B^{\pm}_{\alpha}$ be the group generated by $H$ and $U_{\alpha}$; $G_{\alpha}$ be the group generated by $B^{\pm}_{\alpha}$; and $B^{\pm}$ be the group generated by $B_{\alpha}$ for all $\alpha\in \Delta^{re}_{\pm}$. 
 
\noindent The following properties of these subgroups can be verified easily.
 \begin{enumerate}[before=\itshape,font=\normalfont]
\item[ (RD1)] Let $\alpha, \beta$ be a prenilpotent pair then $[U_{\alpha},U_{\beta}]\subset \langle\; U_{\gamma}\mid \gamma\in ]\alpha,\beta[\; \rangle$.
\item[ (RD2)]  For each $i\in I$, $B^{+}_{\alpha_{i}}\cap B^{-}_{\alpha_{i}}=H$.
 \item[ (RD3)] The group $B^{+}_{\alpha_{i}}$ has two double cosets in $G_{\alpha_{i}}$.
\item[ (RD4)] For each $i\in I$ and $\beta\in \Delta^{re}$, there exists an element $s_{i}\in G_{\alpha_{i}}$ such that $s_{i}B_{\alpha}s_{i}^{-1}=B_{w_{i}\alpha}$.
 \item[ (RD5)]  For each $i\in I$, $B^{+}_{\alpha_{i}}$ is not contained in $B^{-}$ and $B^{-}_{\alpha_{i}}$ is not contained in $B^{+}$.
\end{enumerate}
Let $N$ be the subgroup generated by $H$ and $\tilde{w}_{i}$ for all $i\in I$. For a proof of the next result, we refer readers to \cite[Section 5]{Ti9}.
\begin{theorem}
The axioms RD1-RD5 imply the following important consequences,
\begin{enumerate}[before=\itshape,font=\normalfont]
\item[(a)] The pairs $(B^{+}, N)$ and $(B^{-}, N)$ form a Tits system.
\item[(b)] There exists a unique homomorphism $\phi\colon N\longrightarrow W$ 
with $Ker\phi= H$, and for all $n\in N$ and $\beta\in \Delta^{re}$, $nB_{\beta}n^{-1}=B_{\phi(n)\beta}$.
\item[(c)] Group $G$ has {\it Bruhat decompositions} 
\begin{eqnarray*}
G&=&BNB=B^{-}NB^{-}\\
&=&BNB^{-}=B^{-}NB.
\end{eqnarray*}
\end{enumerate}
\end{theorem}

 \subsection{Subgroups and Decompositions}\label{subdec}
 
Recall the notation on non-archimedean local field $\mathcal{K}$ from Subsection~\ref{sub3}. From now on, we consider Carbone-Garland Kac-Moody group $G=G(\mathcal{K})$ over $\mathcal{K}$. The group $G$ has an integral subgroup
which is defined as
\begin{eqnarray}
K:=\{g\in G\mid gV_{\mathcal{O}}\subset V_{\mathcal{O}}\},
\end{eqnarray}
where 
\begin{eqnarray}\label{repintlattice}
V_{\mathcal{O}}= V_{\mathbb{Z}}\otimes \mathcal{O}.
\end{eqnarray}
 The group $K$ is an analogue of maximal compact subgroup from the finite dimensional theory. There is a pair of {\it unipotent subgroups}
\begin{eqnarray*}
U^{\pm}=\langle U_{\alpha}\mid \alpha\in \Delta^{re,\pm}\rangle 
\end{eqnarray*}
Let $U^{\pm}_{\mathcal{O}}=U^{\pm}\cap K$ and $H_{\mathcal{O}}=H\cap K$ be the integral subgroups. The weight lattice $\Lambda^{\vee}$ can be identified with $H^{'}=H/H_{\mathcal{O}}$ via the map $\lambda^{\vee}\mapsto \pi^{\lambda^{\vee}}$, and $G$ has an {\it Iwasawa decomposition} 
\begin{eqnarray}
G=\cup_{\mu^{\vee}\in \Lambda^{\vee}}K\pi^{\mu^{\vee}}U^{+}=\cup_{\nu^{\vee}\in \Lambda^{\vee}}U^{+}\pi^{\nu^{\vee}}K
\end{eqnarray}
 with respect to $U^{+}$ and
  \begin{eqnarray}
  G=\cup_{\gamma^{\vee}\in \Lambda^{\vee}}K\pi^{\gamma^{\vee}}U^{-}=\cup_{\delta^{\vee}\in \Lambda^{\vee}}U^{-}\pi^{\delta^{\vee}}K
  \end{eqnarray} 
  with respect to $U^{-}$. Let $G(\mathrm{k})$ be the Kac-Moody group over the finite residue field $\mathrm{k}$ and 
 $\varpi:K\longrightarrow G(\mathrm{k})$
 be the reduction mod $\pi$ map. Set
 \begin{eqnarray}
 G_{\pi}:=\{g\in K\mid \varpi(g)=1\}
 \end{eqnarray} 
 and $U^{\pm}_{\pi}:=U^{\pm}\cap G_{\pi}$.
 
 The group $K$ has a pair of subgroups defined as 
 \begin{eqnarray}
 I^{\pm}=\varpi^{-1}(B^{\pm}(\mathrm{k})).
 \end{eqnarray}
 These groups are known as the {\it Iwahori subgroups} and admit the following direct product decompositions,
 \begin{eqnarray*}
 I^{+}&=&U^{+}_{\mathcal{O}}U^{-}_{\pi}H_{\mathcal{O}}\\
  I^{-}&=&U^{-}_{\mathcal{O}}U^{+}_{\pi}H_{\mathcal{O}},
 \end{eqnarray*}
 which are known as the {\it Iwahori-Matsumoto decompositions}. The group 
$K$ admits the following decompositions
\begin{eqnarray*}
K &=&\cup_{w\in W}I^{+}wI^{+}=\cup_{w\in W}I^{-}wI^{+}\\
&=&\cup_{w\in W}I^{+}wI^{-}=\cup_{w\in W}I^{-}wI^{-}
\end{eqnarray*}
known as the {\it Iwahori decompositions}.

For $w\in W$, we define the following two subsets of the set of roots $\Delta$.
\begin{eqnarray*}
 S^{+}_{w}&:=&\{\alpha\in \Delta_{+}\mid w\alpha\in \Delta_{-}\}=\Delta_{+}\cap w^{-1}(\Delta_{-}),\\
 S^{-}_{w}&:=&\{\alpha\in \Delta_{-}\mid w\alpha\in \Delta_{+}\}=\Delta_{-}\cap w^{-1}(\Delta_{+}).
 \end{eqnarray*}
 Similarly, we define the subsets $S^{\pm,\vee}_{w}\subset \Delta^{\vee}$. By using $S^{\pm}_{w}$, we introduce finitely generated subgroups 
 $$U^{\pm}_{w}=\langle U_{\alpha}\mid \alpha\in S^{\pm}_{w}\rangle.$$ 
Let $U^{\pm, w}=U^{\pm}-U^{\pm}_{w}$. Set
\begin{eqnarray}
U^{\pm}_{w, \mathcal{O}}&=&U^{\pm}_{ w}\cap U^{\pm}_{\mathcal{O}}, \;\;U^{\pm}_{w, \pi}=U^{\pm}_{ w}\cap U^{\pm}_{\pi}\label{congsub1}\\
 U^{\pm, w}_{\mathcal{O}}&=&U^{\pm, w}\cap U^{\pm}_{\mathcal{O}},\;\; U^{\pm, w}_{\pi}=U^{\pm, w}\cap U^{\pm}_{\pi}.
 \end{eqnarray}

\subsection{Some Representation Theoretic Results}\label{represntation}

For $i\in I$, let $\omega_{i}\in \mathfrak{h}^{\ast}$ be a {\it fundamental weight} defined by $\omega_{i}(\alpha^{\vee}_{j})=\delta_{ij}$ for all $i,j\in I$ and $\omega_{i}=0$ outside $Q^{\vee}\otimes_{\mathbb{Z}}\mathbb{C}$ (where $\delta_{ij}$ is the Kronecker delta) and suppose
 \begin{eqnarray}\label{rhodef}
 \rho=\sum_{i\in I}\omega_{i}.
 \end{eqnarray}
 Let $V^{\rho}$ be the highest highest weight representation with the highest weight $\rho$ and a fixed highest weight vector $v_{\rho}$. Let $w\in W$ and $F^{m}$ be the canonical filtration on the universal enveloping algebra of $\mathcal{U}( \mathfrak{n}^{+})$. The following lemma from \cite[Section 18]{BFG} relates the weight vectors $v_{\rho}$ and $v_{w\rho}:=wv_{\rho}$ in $V^{\rho}$,
 \begin{lemma}[Joseph's Lemma]\label{Josehphlemma}
 Suppose $v_{\rho}\in F^{m}(\mathcal{U}( \mathfrak{n}^{+}))v_{w\rho}$. Then we must have $\ell(w)<2m$, where $\ell(w)$ denotes the length of $w$.
 \end{lemma}
\noindent A useful consequence of the above lemma is the following corollary from \cite[P. 51]{BGKP} which is proven for affine case originally, but also holds true in general Kac-Moody settings.
 \begin{corollary}\label{usefulcorrolary}
 If $w\in W$ and $\mu^{\vee}\in Q^{\vee}_{-}$ be such that $U^{-}\cap U^{+}_{w,\pi}w\pi^{\mu^{\vee}}U^{+}\ne \emptyset $ then $l(w)\le2 \langle \rho, \mu^{\vee}\rangle$.
 \end{corollary}

\section{Approximation Theorem}
  In this section we shall prove Theorem~\ref{maintheorem1}. This theorem establishes a link between the image $\mathcal{S}_{\lambda^{\vee}}$ of the Satake isomorphism and the Gindikin-Karpelevich sum $\mathscr{G}_{\lambda^{\vee}}$ when $\lambda^{\vee}$ is very large. 
 
 \subsection{Proof of Main Theorem}
 Let $\rho\in\Lambda$ be the element as defined in (\ref{rhodef}). Throughout this section we fix our highest weight module $V^{\rho}$ with highest weight $\rho$ and equipped with the norm $||.||$ given in (\ref{repthnorm4}). We also fix a primitive highest weight vector $v_{\rho}$ in $V^{\rho}$, that is, $||v_{\rho}||=1$. Let $P_{\rho}$ be the set of weights $V^{\rho}$ and for $\nu\in P_{\rho}$, $\eta_{\nu}$ be the projection map~(\ref{projmap4}). First, we prove the following lemma which will be used to prove the statement of the theorem. 
 
 \begin{lemma}\label{apptheL2}
Let $\mu^{\vee}$ be fixed and $\lambda^{\vee}\in\Lambda_{+}$ be regular. There exists a finite subset $\Xi=\Xi(\lambda^{\vee},\mu^{\vee})\subset P_{\rho}$ such that if $u^{-}\in U^{-}$ satisfies
\begin{eqnarray}\label{CondL2}
\pi^{\lambda^{\vee}}u^{-}\in K\pi^{\lambda^{\vee}-\mu^{\vee}}U^{+},
\end{eqnarray}
and $\eta_{\nu}(u^{-}v_{\rho})\not\in V^{\rho}_{\mathcal{O}}$ for some $\nu\in P_{\rho}$, then $\nu\in \Xi$.
\proof  If $u^{-}\in U^{-}_{\mathcal{O}}$ then $u^{-}v_{\rho}\in V^{\rho}_{\mathcal{O}}$ and hence there is nothing to prove. Let $U^{-}(\lambda^{\vee}, \mu^{\vee})$ be the set of elements $u^{-}\in U^{-}\setminus U^{-}_{\mathcal{O}}$ which satisfy (\ref{CondL2}), and set
$$\Sigma=\Sigma(\lambda^{\vee},\mu^{\vee})=\{\gamma\in Q_{+}\mid \eta_{\rho-\gamma}(u^{-}v_{\rho})\not\in V^{\rho}_{\mathcal{O}},\; u^{-}\in U^{-}(\lambda^{\vee}, \mu^{\vee})  \}.$$ 
It suffices to show that $\Sigma$ is a finite set. Let $u^{-}\in U^{-}(\lambda^{\vee}, \mu^{\vee})$ with a corresponding $\gamma\in \Sigma$ such that $\gamma=\sum_{i=1}^{l}k_{i}\alpha_{i}$, $k_{i}\in \mathbb{Z}_{>0}$ and $\alpha_{i}\in\Pi$. 
By assumption 
\begin{eqnarray}\label{EQb}
\pi^{\lambda^{\vee}}u^{-}= k\pi^{\lambda^{\vee}-\mu^{\vee}}u^{+},
\end{eqnarray}
for some $k\in K$ and $u^{+}\in U^{+}$.
We apply both sides of (\ref{EQb}) to the highest weight vector $v_{\rho}$. The action of the left hand side of (\ref{EQb}), Lemma~\ref{sec1l1} and the fact $\eta_{\rho-\gamma}(u^{-}v_{\rho})\not\in V^{\rho}_{\mathcal{O}}$ give,  
\begin{eqnarray}\label{EQb1}
 ||\pi^{\lambda^{\vee}} u^{-}v_{\rho}||&\ge &  ||\pi^{\lambda^{\vee}} v_{\rho-\gamma}||> q^{-\langle \rho-\gamma, \lambda^{\vee} \rangle}.
\end{eqnarray}
The right hand side of (\ref{EQb}) acts as,
\begin{eqnarray}\label{EQb2}
 ||k\pi^{\lambda^{\vee}-\mu^{\vee}}u^{+}v_{\rho}||&= &  q^{-\langle \rho, \lambda^{\vee}-\mu^{\vee}\rangle }.
\end{eqnarray}
So, (\ref{EQb1}) and  (\ref{EQb2}) imply $q^{-\langle\rho-\gamma,\lambda^{\vee}\rangle}< q^{-\langle\rho,\lambda^{\vee}-\mu^{\vee}\rangle}. $
Which shows $ q^{\langle\gamma,\lambda^{\vee}\rangle}< q^{\langle\rho,\mu^{\vee}\rangle}$
and hence $ {\langle\gamma,\lambda^{\vee}\rangle}< {\langle\rho,\mu^{\vee}\rangle}.$
Consequently,
\begin{eqnarray}\label{EQc}
\sum_{i=1}^{l}k_{i}{\langle\alpha_{i},\lambda^{\vee}\rangle}&< &{\langle\rho,\mu^{\vee}\rangle}. 
\end{eqnarray}
Since $\mu^{\vee}$ is fixed, and $\lambda^{\vee}$ is dominant and regular, ${\langle\alpha_{i},\lambda^{\vee}\rangle}$ are fixed positive numbers for $1\le i\le l$, the bound  ${\langle\rho,\mu^{\vee}\rangle}$ in (\ref{EQc}) on the coefficients $k_{i}$ appearing in the simple root decomposition of $\gamma$ implies that we have only finitely many choices for $\gamma\in Q_{+}$. Therefore, the set $\Sigma\subset Q^{\vee}$ is finite and this completes the proof.
\end{lemma}
Now, we give the proof of the Approximation Theorem.
\proof[Proof of Theorem~\ref{maintheorem1}]\let\qed\relax
To prove the assertion of the theorem, we show that for $\mu^{\vee}\in Q_{+}^{\vee}$ and $\lambda^{\vee}$ sufficiently dominant the following set theoretic inclusions hold
\begin{eqnarray}
K\pi^{\lambda^{\vee}}K\cap K\pi^{\lambda^{\vee}-\mu^{\vee}}U^{+}&\subset &K\pi^{\lambda^{\vee}}U^{-}\cap K\pi^{\lambda^{\vee}-\mu^{\vee}}U^{+}\label{limitcont1}\\
K\pi^{\lambda^{\vee}}U^{-}\cap K\pi^{\lambda^{\vee}-\mu^{\vee}}U^{+}& \subset &K\pi^{\lambda^{\vee}}K\cap K\pi^{\lambda^{\vee}-\mu^{\vee}}U^{+}\label{limitcont2}.
\end{eqnarray}
The first containment follows exactly as it does in affine case \cite[Subsection 6.3]{BGKP}, we only sketch its (slightly modified) proof here. First, note that if $\lambda^{\vee}$ is dominant then 
\begin{eqnarray}\label{lamdomint1}
\pi^{\lambda^{\vee}}U^{+}_{\mathcal{O}}\pi^{-\lambda^{\vee}}\subset U^{+}_{\mathcal{O}}
\end{eqnarray}
and $K\pi^{\lambda^{\vee}}I^{+}= K\pi^{\lambda^{\vee}}U^{-}_{\pi}\subset K\pi^{\lambda^{\vee}}U^{-}.$

Thus, it suffices to show that for $\mu^{\vee}\in Q^{\vee}_{+}$ and $\lambda^{\vee}$ as above 
\begin{eqnarray}
K\pi^{\lambda^{\vee}}K\cap K\pi^{\lambda^{\vee}-\mu^{\vee}}U^{+}\subset K\pi^{\lambda^{\vee}}I^{+}.
\end{eqnarray}
 For this, let $k_{1}\in K$ be such that 
\begin{eqnarray}\label{appthe1}
\pi^{\lambda^{\vee}}k_{1}\in  K\pi^{\lambda^{\vee}-\mu^{\vee}}U^{+}
\end{eqnarray}
 and suppose $k_{1}\in I^{+}wI^{+}$ for some $w\in W$. For any $w\in W$, $U^{-}_{\pi}wI^{+}\subseteq wI^{+}$. Then by using this fact and (\ref{lamdomint1}), we have
\begin{eqnarray}
\pi^{\lambda^{\vee}}k_{1}\in \pi^{\lambda^{\vee}} I^{+}wI^{+}&=&\pi^{\lambda^{\vee}}U^{+}_{\mathcal{O}}U^{-}_{\pi}wI^{+}\nonumber\\
&\subseteq& U^{+}_{\mathcal{O}}\pi^{\lambda^{\vee}}wI^{+}\nonumber\\
&=& U^{+}_{\mathcal{O}}\pi^{\lambda^{\vee}}wU^{-}_{\pi}U^{+}_{\mathcal{O}}\label{lamdomint2}.
\end{eqnarray}\label{limitcont3}
So, (\ref{appthe1}) and (\ref{lamdomint2}) imply that for $u_{1},u_{2}\in U^{+}_{\mathcal{O}}$, $u^{-}\in U^{-}_{\pi}$, $u_{3}\in U^{+}$ and $k'\in K$
\begin{eqnarray}
\pi^{\lambda^{\vee}}k_{1}=u_{1}\pi^{\lambda^{\vee}}wu^{-}u_{2}=k'\pi^{\lambda^{\vee}-\mu^{\vee}}u_{3}
\end{eqnarray}
Thus, by taking $(u_{1})^{-1}k'=k\in K$ and $u_{3}u_{2}^{-1}=u^{+}\in U^{+}$, we have
\begin{eqnarray}\label{limitcont4}
\pi^{\lambda^{\vee}}wu^{-}=k\pi^{\lambda^{\vee}-\mu^{\vee}}u^{+}.
\end{eqnarray}
Next, we choose $\lambda^{\vee}\in\Lambda^{\vee}$ sufficiently dominant such that if $\sigma\in W$, $\sigma\ne 1$ and $\sigma\lambda^{\vee}=\lambda^{\vee}-\beta^{\vee}$ for some $\beta^{\vee}\in Q_{+}$, then 
\begin{eqnarray}\label{limitcont5}
\langle \rho, \beta^{\vee}\rangle > \langle\rho, \mu^{\vee}\rangle.
\end{eqnarray}
If $w\ne 1$, by letting the both sides of (\ref{limitcont4}) act on the highest weight vector $v_{\rho}$, we compute
\begin{eqnarray}\label{limitcont6}
q^{-\langle \rho, \lambda^{\vee}- \mu^{\vee}\rangle}&=&||k\pi^{\lambda^{\vee}-\mu^{\vee}}u^{+}v_{\rho}|| =||\pi^{\lambda^{\vee}}wu^{-}v_{\rho}||\nonumber\\
& \ge&||\pi^{\lambda^{\vee}}v_{w\rho}||= q^{-\langle w\rho, \lambda^{\vee}\rangle}=q^{-\langle\rho ,w^{-1}\lambda^{\vee}\rangle}.
\end{eqnarray}
This implies $\langle \rho, \lambda^{\vee}- \mu^{\vee}\rangle\le \langle\rho ,w^{-1}\lambda^{\vee}\rangle$ and this results in a contradiction of the inequality (\ref{limitcont5}).

For the containment (\ref{limitcont2}), let $u^{-}\in U^{-}$ be such that $\pi^{\lambda^{\vee}}u^{-}\in  K\pi^{\lambda^{\vee}-\mu^{\vee}}U^{+}$. If $u^{-}\in U^{-}_{\mathcal{O}}$, then our theorem follows.  If $u^{-}\not\in U^{-}_{\mathcal{O}}$ then by Lemma~\ref{apptheL2}, there exists a finite subset $\Xi\subset P_{\rho}$ such that if $\nu\in \Xi$ then $\eta_{\nu}(u^{-}v_{\rho})\not\in V^{\rho}_{ \mathcal{O}}$. Moreover, as in the proof of the above lemma, if we write $\nu=\rho-\gamma$ for some $\gamma\in Q_{+}$, we get $q^{\langle\gamma,\lambda^{\vee}\rangle}\le q^{\langle\rho,\mu^{\vee}\rangle}.$
Thus by choosing $\lambda^{\vee}$ sufficiently dominant corresponding to the finitely many elements in $\Xi$, we can arrange
\begin{eqnarray}
q^{\langle\rho,\mu^{\vee}\rangle}< q^{\langle\gamma,\lambda^{\vee}\rangle},
\end{eqnarray}
leading us to a contradiction.
\subsection{Iwahori Refinement }
\noindent The following proposition is an Iwahori analogue of the Approximation Theorem.

\begin{proposition}\label{Prop1}
Let $w\in W$ and $\mu^{\vee}\in \Lambda^{\vee}$ be fixed. Then for all sufficiently dominant $\lambda^{\vee}=\lambda^{\vee}(\mu^{\vee}, w)$ (that is, sufficiently dominant depending on $\mu^{\vee}$ and $w$), we have 
\begin{eqnarray*}
I^{-}\pi^{\lambda^{\vee}}U^{-}\cap I^{-}w\pi^{\lambda^{\vee}-\mu^{\vee}}U^{+}&=& I^{-}\pi^{\lambda^{\vee}}U^{-}_{\mathcal{O}}\cap I^{-}w\pi^{\lambda^{\vee}-\mu^{\vee}}U^{+}.
\end{eqnarray*}

\proof One inclusion is straightforward. So, we prove the other
\begin{eqnarray}\label{EQa}
I^{-}\pi^{\lambda^{\vee}}U^{-}\cap I^{-}w\pi^{\lambda^{\vee}-\mu^{\vee}}U^{+} &\subset & I^{-}\pi^{\lambda^{\vee}}U^{-}_{\mathcal{O}}\cap I^{-}w\pi^{\lambda^{\vee}-\mu^{\vee}}U^{+}.
\end{eqnarray}
For this, let $v^{-}\in U^{-}$ be such that $\pi^{\lambda^{\vee}}v^{-}\in  I^{-}w\pi^{\lambda^{\vee}-\mu^{\vee}}U^{+}$. Then using the Iwahori-Matsumoto decomposition $I^{-}=U_{\pi}^{+}U^{-}_{\mathcal{O}}H_{\mathcal{O}}$, we have
$$\pi^{\lambda^{\vee}}v^{-}\in U^{-}_{\mathcal{O}}U^{+}_{w,\pi}w\pi^{\lambda^{\vee}-\mu^{\vee}}U^{+}.$$
Hence, for the containment (\ref{EQa}), it suffices to show that:\\
\noindent ({\bf P1}) Let $w\in W$, $\mu^{\vee}\in \Lambda^{\vee}$ be fixed. Then for sufficiently dominant $\lambda^{\vee}=\lambda^{\vee}(\mu^{\vee}, w)$, if $\pi^{\lambda^{\vee}}u^{-}\in U^{+}_{w,\pi}w\pi^{\lambda^{\vee}-\mu^{\vee}}U^{+}$ with $u^{-}\in U^{-}$, then $u^{-}\in U^{-}_{\mathcal{O}}$.\\

\noindent To prove ({\bf P1}), let $u^{-}\in U^{-}$ be such that
\begin{eqnarray}\label{EQP1}
\pi^{\lambda^{\vee}}u^{-}&=&u^{+}_{w}w\pi^{\lambda^{\vee}-\mu^{\vee}}u^{+},
\end{eqnarray}
for some $u^{+}_{w}\in U^{+}_{w,\pi}$ and $u^{+}\in U^{+}$.
We apply both sides of (\ref{EQP1}) to the highest weight vector  $v_{\rho}$. The right hand side gives us,
\begin{eqnarray*}
u^{+}_{w}w\pi^{\lambda^{\vee}-\mu^{\vee}}u^{+}v_{\rho}&=&u^{+}_{w}w\pi^{\lambda^{\vee}-\mu^{\vee}}v_{\rho}\\
&=&\pi^{\langle\rho,\lambda^{\vee}-\mu^{\vee}\rangle}u^{+}_{w}wv_{\rho}.
\end{eqnarray*}

Since $u^{+}_{w}\in K$, $||u^{+}_{w}||=1$ and hence
\begin{align}
 ||u^{+}_{w}w\pi^{\lambda^{\vee}-\mu^{\vee}}u^{+}v_{\rho}||=||\pi^{\langle\rho,\lambda^{\vee}-\mu^{\vee}\rangle}wv_{\rho}||\nonumber\\
 =q^{-\langle\rho,\lambda^{\vee}-\mu^{\vee}\rangle}\label{EQP2}.
 \end{align}
On the other hand, the element on the left hand side of (\ref{EQP1}) acts as 
\begin{eqnarray}\label{EQ3}
\pi^{\lambda^{\vee}}u^{-}v_{\rho}&=&\pi^{\lambda^{\vee}}(\sum_{\nu\in P_{\rho}}v_{\nu}),
\end{eqnarray}
If $u^{-}\not\in U^{-}_{\mathcal{O}}$, then there exists at least one weight $\xi:=\rho-\gamma\in P_{\rho} $, with $\gamma\in Q_{+}$ such that the corresponding weight vector $v_{\rho-\gamma}$ on the right hand side of (\ref{EQ3}) is not integral, i.e. $\eta_{\rho-\gamma}(u^{-}v_{\rho})\not\in V_{\rho, \mathcal{O}}$. This gives
\begin{align}
 ||\pi^{\lambda^{\vee}}u^{-}v_{\rho}||\ge ||\pi^{\lambda^{\vee}}v_{\rho-\gamma}||\nonumber\\
 \ge q^{-\langle\rho-\gamma,\lambda^{\vee}\rangle}.\label{EQ4}
\end{align}
So,  (\ref{EQP2}) and (\ref{EQ4}) imply
\begin{eqnarray}\label{Ieq1}
q^{\langle\rho,\mu^{\vee}\rangle}&\ge& q^{\langle\gamma,\lambda^{\vee}\rangle}.
\end{eqnarray}

\begin{claim}
 There are finitely many $\gamma\in Q_{+}$ such that $\eta_{\rho-\gamma}(u^{-}v_{\rho})\ne 0$ for all $u^{-}$ satisfying (\ref{EQP1}).
\proof The subgroup $U^{+}_{w,\pi}$ is generated by the finite number of root subgroups $U_{\alpha,\pi}$, $\alpha\in S^{+}_{w^{-1}}$. So, for each element  $u^{+}_{w}\in U^{+}_{w,\pi}$, when $u^{+}_{w}w$ acts on the highest weight vector $v_{\rho}$, there are a finite number of choices of weights that can appear in the weight vector decomposition of $u^{+}_{w}wv_{\rho}$. The same is true for $u^{-}$ appearing on the left hand side of (\ref{EQP1}). Hence our claim holds.
\end{claim}

\noindent Thus, by choosing $\lambda^{\vee}$ sufficiently dominant corresponding to these finitely many $\gamma$ from the claim, we may arrange
$$q^{\langle\rho,\mu^{\vee}\rangle}< q^{\langle\gamma,\lambda^{\vee}\rangle}.$$
Thus our assumption $u^{-}\not\in U^{-}_{\mathcal{O}}$ leads to a contradiction of (\ref{Ieq1}), and consequently the statement  ({\bf P1}) and the proposition follow.
\end{proposition}
\section{Iwahori Level Decomposition}\label{section4}
Now, we initiate the proof of the Theorem~\ref{maintheorem4}. For $\lambda^{\vee},\mu^{\vee}\in\Lambda^{\vee}$, set 
\begin{eqnarray}
M(\lambda^{\vee}, \mu^{\vee})=K\backslash K\pi^{\mu^{\vee}}U^{+} \cap K\pi^{\lambda^{\vee}}K.
\end{eqnarray}
We begin by establishing a bijective correspondence between the coset space $M(\lambda^{\vee},\mu^{\vee})$ and a disjoint union of so-called Iwahori pieces. This disjoint union is indexed by the Weyl group $W$ and the main result of this section asserts that there are finitely many elements of $W$ which contribute in this union. 

\subsection{Iwahori Pieces}\label{iwahoripiecesx1}
Let $\Gamma\le G$, $X$ be a right-$\Gamma$ and $Y$ be a left-$\Gamma$ set. We need the following relation on the set $X\times Y$ from \cite[Section 4]{BKP}.
\begin{definition}
Let $(x, y), (x',y')\in X\times Y$, $(x, y)\sim (x',y')$  if and only if there exists some $r\in \Gamma$ such that, $x'=xr$ and  $y'=r^{-1}y$. 
\end{definition} 
\noindent One can check that $\sim$ is an equivalence relation. We denote by $X\times_{\Gamma} Y$, the quotient space $(X\times Y)/\sim$. 

For $\lambda^{\vee}\in \Lambda^{\vee}$, we take $X=U^{+} K$, $Y=K\pi^{\lambda^{\vee}}K$, $\Gamma=K$ and consider the following map induced by the multiplication
\begin{eqnarray}
m_{\lambda^{\vee}}&:&U^{+} K\times_{K}K\pi^{\lambda^{\vee}}K\longrightarrow G.\label{multmap2}
\end{eqnarray}
For $w\in W$ and $\lambda^{\vee}\in \Lambda^{\vee}$, we take $X=U^{+}w I^{-}$, $Y=I^{-}\pi^{\lambda^{\vee}}K$, $\Gamma=I^{-}$ and consider also the following map induced by multiplication:
\begin{eqnarray}
m_{w,\lambda^{\vee}}&:&U^{+}w I^{-}\times_{I^{-}}I^{-}\pi^{\lambda^{\vee}}K\longrightarrow G.\label{multmap2}
\end{eqnarray}
As in the affine case from \cite[Section  4.4]{MP2}, it can be shown that for $\mu^{\vee}\in \Lambda^{\vee}$, the fiber $m_{\lambda^{\vee}}^{-1}(\pi^{\mu^{\vee}})$  is in bijective correspondence with $M(\lambda^{\vee},\mu^{\vee})$. Also, the following lemma can be proven more generally along the same lines as those of \cite[Lemma 7.3.3]{BKP}, which again was written in the affine context.

\begin{lemma}\label{lemiwahoridec}
 For all $w\in W$ and $\lambda^{\vee}\in \Lambda^{\vee}_{+}$ regular, the fibers $m_{w,\lambda^{\vee}}^{-1}(\pi^{\mu^{\vee}})$ are disjoint and there is a bijection
 \begin{eqnarray}\label{Iwsum11}
m_{\lambda^{\vee}}^{-1}(\pi^{\mu^{\vee}})\simeq\sqcup_{w\in W}m_{w,\lambda^{\vee}}^{-1}(\pi^{\mu^{\vee}}).
\end{eqnarray}
\end{lemma}
\noindent From now on, for $w\in W$ and $\mu^{\vee}\in\Lambda^{\vee}$, the fiber $m_{w,\lambda^{\vee}}^{-1}(\pi^{\mu^{\vee}})$ will be referred to as an Iwahori piece of $M(\lambda^{\vee}, \mu^{\vee})$. For each $w\in W$ and fixed $\lambda^{\vee}, \mu^{\vee}\in \Lambda^{\vee}$, by definition
\begin{eqnarray}\label{fibersa1} 
m_{\lambda^{\vee},w}^{-1}(\pi^{\mu^{\vee}}) &=&I^{-}\backslash I^{-}w^{-1}\pi^{\mu^{\vee}}U^{+}\cap I^{-}\pi^{\lambda^{\vee}}K.
\end{eqnarray}

\noindent Theorem~\ref{maintheorem4} will follow if we prove: 
\begin{itemize}
\item[(a)] for fixed $ \mu^{\vee}\in \Lambda^{\vee}$ and $ \lambda^{\vee}$ regular and sufficiently dominant, there are finitely many elements of $W$ which contribute in the union on the right hand side of (\ref{Iwsum11}),
\item[(b)] for each such $w\in W$, the Iwahori piece $m_{w,\lambda^{\vee}}^{-1}(\pi^{\mu^{\vee}})$ is finite.
\end{itemize}
 Part (b) will be proven in Section~\ref{Finitesupport} and \ref{sectionfiberfin} and part (a) follows by the following proposition.

\begin{proposition}\label{Wfin}
For a fixed $\mu^{\vee}\in \Lambda^{\vee}$, $\lambda^{\vee}\in \Lambda^{\vee}$ regular and sufficiently large with respect to $\mu^{\vee}$, there exists a finite subset $\Omega=\Omega(\lambda^{\vee}, \mu^{\vee})\subset W$ such that 
$$m_{\lambda^{\vee}}^{-1}(\pi^{\mu^{\vee}})\simeq\sqcup_{w\in \Omega}m_{w,\lambda^{\vee}}^{-1}(\pi^{\mu^{\vee}}).$$
\proof By (\ref{fibersa1}),  it suffices to show the following: for $\mu^{\vee}$ fixed and $\lambda^{\vee}$ regular and sufficiently large, there exist finitely many $w\in W$ such that   
\begin{eqnarray}\label{finw1}
I^{-}w\pi^{\mu^{\vee}}U^{+}\cap I^{-}\pi^{\lambda^{\vee}}K\ne\emptyset.
\end{eqnarray}
We replace $K$ by $K=\cup_{\sigma\in W}I^{+}\sigma  I^{-}$ and then use the Iwahori-Matsumoto decomposition $I^{+}=U^{+}_{\mathcal{O}}U^{-}_{\pi}H_{\mathcal{O}}$ on the left hand side of (\ref{finw1}) to obtain
\begin{eqnarray}\label{finwnew2}
I^{-}w\pi^{\mu^{\vee}}U^{+}\cap I^{-}\pi^{\lambda^{\vee}}K&=& \cup_{\sigma\in W}I^{-}w\pi^{\mu^{\vee}}U^{+}\cap I^{-}\pi^{\lambda^{\vee}}U^{+}_{\mathcal{O}}U^{-}_{\pi}\sigma  I^{-}\nonumber\\
&=& \cup_{\sigma\in W}I^{-}w\pi^{\mu^{\vee}}U^{+}\cap I^{-}\pi^{\lambda^{\vee}}\sigma  I^{-},
\end{eqnarray}
where in the last step we use the fact that if $\lambda^{\vee}$ is dominant and regular then $\pi^{\lambda^{\vee}}U^{+}_{\mathcal{O}}\pi^{-\lambda^{\vee}}\subset U^{+}_{\pi}$. Consider $\sigma\in W$ such that
\begin{eqnarray} 
I^{-}w\pi^{\mu^{\vee}}U^{+}\cap I^{-}\sigma\pi^{\sigma\lambda^{\vee}}U^{-}_{\mathcal{O}}\ne\emptyset.
\end{eqnarray} 
Now, 
\begin{eqnarray} 
I^{-}w\pi^{\mu^{\vee}}U^{+}\cap I^{-}\sigma\pi^{\sigma\lambda^{\vee}}U^{-}_{\mathcal{O}}\subset K \pi^{\mu^{\vee}}U^{+}\cap K \pi^{\sigma\lambda^{\vee}}U^{-}.
\end{eqnarray} 
 By the second part of Theorem~\ref{maintheorem2} we get,
\begin{eqnarray}\label{finw2}
\mu^{\vee}\le \sigma\lambda^{\vee}.
\end{eqnarray}
Since $\lambda^{\vee}$ is regular,  if we choose $\lambda^{\vee}$ very large compared to $\mu^{\vee}$ then (\ref{finw2}) holds only for $\sigma=1$. Hence (\ref{finwnew2}) implies 
$$I^{-}w\pi^{\mu^{\vee}}U^{+}\cap I^{-}\pi^{\lambda^{\vee}}K=I^{-}w\pi^{\mu^{\vee}}U^{+}\cap I^{-}\pi^{\lambda^{\vee}}U^{-}_{\mathcal{O}}.$$ 
By (\ref{finw1}), we also have:
\begin{eqnarray*}
U^{-}\cap U^{+}_{w^{-1},\pi}w\pi^{\mu^{\vee}-\lambda^{\vee}}H_{\mathcal{O}}U^{+}\ne\emptyset.
\end{eqnarray*} 
 Finally, Corollary~\ref{usefulcorrolary} implies 
 \begin{eqnarray}\label{boundonw}
 l(w)\le 2\langle \rho, \lambda^{\vee}-\mu^{\vee}\rangle . 
 \end{eqnarray}
 The bound (\ref{boundonw}) proves the Proposition.
 \end{proposition}

\section{An Integral and Recursion Relation}\label{Finitesupport}
In this section, we introduce two propotional integrals $I_{w,\lambda^{\vee}}$ and $\tilde{I}_{w,\lambda^{\vee}}$, and prove the convergence of $I_{w,\lambda^{\vee}}$ by showing that it satisfies a recursion relation in terms of a certain operator. This will imply the convergence of $ \tilde{I}_{w,\lambda^{\vee}}$ as well. The finiteness of the level sets of $\tilde{I}_{w,\lambda^{\vee}}$ obtained as a consequence of this convergence will be used to obtain the proof of the Weak Spherical Finiteness in the next section.
 
 \subsection{The Integrals }\label{Int}
Let $\rho$ be the sum of fundamental weights as introduced earlier (see  (\ref{rhodef})).
\begin{definition} We define a function
$$\Phi_{\rho}:G\longrightarrow \mathbb{C}[\Lambda^{\vee}]$$
by the formula $\Phi_{\rho}(g)=q^{-\langle \rho,\mu^{\vee}\rangle}e^{\mu^{\vee}},$ where $g\in G$ has an Iwasawa decomposition $g\in U\pi^{\mu^{\vee}}K$. 
\end{definition}
For $w\in W$, recall the subset $S^{-}_{w^{-1}}$ and the corresponding subgroups $U^{-}_{w^{-1}}$, $U^{-}_{w^{-1},\mathcal{O}}$ and $U^{-}_{w^{-1},\pi}$ from Subsection~\ref{subdec}. The group $U^{-}_{w^{-1}}$ is finite dimensional and carries a natural Haar measure $du^{-}_{w}$ which is normalized such the group $U^{-}_{w^{-1},\mathcal{O}}$ has volume $1$ with respect to this measure. 
\begin{definition}
The integral $I_{w,\lambda^{\vee}}$ is defined by the following equation,
\begin{eqnarray*}
I_{w,\lambda^{\vee}}&=&\int_{U^{-}_{w^{-1},\pi}}\Phi_{\rho}(u^{-}_{w}\pi^{w\lambda^{\vee}})du^{-}_{w}.
\end{eqnarray*}
\end{definition}
Now, we introduce the integral $\tilde{I}_{w,\lambda^{\vee}}$ as follows:\\
The group $U^{-}_{w^{-1},\pi}$ has finite volume with respect to the measure $du^{-}_{w}$. We normalize the restriction of this measure on $U^{-}_{w^{-1},\pi}$ so that the volume of $U^{-}_{w^{-1},\pi}$ becomes equal to $1$ and call this measure $\tilde{du^{-}_{w}}$. 
\begin{definition}
The integral $\tilde{I}_{w,\lambda^{\vee}}$ is defined by the following equation,
\begin{eqnarray*}
\tilde{I}_{w,\lambda^{\vee}}&=&\int_{U^{-}_{w^{-1},\pi}}\Phi_{\rho}(u^{-}_{w}\pi^{w\lambda^{\vee}})\tilde{du^{-}_{w}}.
\end{eqnarray*}
\end{definition}

\subsection{Demazure-Lusztig operator}
Let $v$ be a formal variable and $\mathbb{C}_{v}:=\mathbb{C}[v]$ be the ring of polynomial in $v$. Let $\mathscr{R}:=\mathbb{C}[[e^{-\alpha_{1}^{\vee}}, e^{-\alpha_{2}^{\vee}},\dots, e^{-\alpha_{l}^{\vee}}]]$ be the ring of power series in indeterminates $e^{-\alpha_{i}^{\vee}}$, $i\in I$. Set
\begin{eqnarray}
\mathscr{L}:=\mathbb{C}_{v}\otimes_{\mathbb{C}}\mathscr{R}
\end{eqnarray}
and 
\begin{eqnarray}
\mathscr{L}[W]:=\{\sum_{w\in W}a_{w}[w]\mid a_{w}\in \mathscr{L}\}.
\end{eqnarray}

Now, we consider another formal variable $X$ and define the following rational functions
\begin{eqnarray*}
{\mathbf b}(X):= \frac{v-1}{1-X}\;\;\text{and}\;\;{\mathbf c}(X):= \frac{1-vX}{1-X}.
\end{eqnarray*}
By expanding ${\mathbf b}(X)$ and ${\mathbf c}(X)$ in $X^{-1}$ and using $X=e^{\alpha^{\vee}}$ for some positive coroot $\alpha^{\vee}$ it follows that ${\mathbf b}(X), {\mathbf c}(X)\in\mathscr{L}$. For a coroot  $\alpha^{\vee}$, we shall denote 
\begin{eqnarray}\label{bandc1}
{\mathbf b}(\alpha^{\vee}):= {\mathbf b}(e^{\alpha^{\vee}})\;\text{and}\;\; {\mathbf c}(\alpha^{\vee}):={\mathbf c}(e^{\alpha^{\vee}}).
\end{eqnarray}
\begin{definition}
For $i\in I$, let $\alpha^{\vee}_{i}$ be the simple coroot and $w_{i}=w_{\alpha_{i}}$ be the simple root reflection. A {\it Demazure-Lusztig operator} on $\mathscr{L}$ is defined by
\begin{eqnarray}\label{bandc2}
{\mathbf T}_{w_{i}}:= {\mathbf c}(\alpha^{\vee}_{i})[w_{i}]+{\mathbf b}(\alpha^{\vee}_{i})[1],
\end{eqnarray}
which, by expanding the rational functions, can be seen to be an element of $\mathscr{L}[W]$.
\end{definition}
\noindent The operator ${\mathbf T}_{w_{i}}$ satisfies the following properties.
\begin{proposition}\label{proplo}
For $i\in I$,
\begin{enumerate}[before=\itshape,font=\normalfont]
\item[(1)]For $i\in I$, ${\mathbf T}_{w_{i}}^{2}=(v-1){\mathbf T}_{w_{i}}+v$.\label{Dellutz1}
\item[(2)] The operators ${\mathbf T}_{w_{i}}$ satisfy the braid relations. So, if $w\in W$ has a reduced decomposition $w=w_{i_{1}}w_{i_{2}}\dots w_{i_{n}}$ then 
$${\mathbf T}_{w}={\mathbf T}_{w_{i_{1}}}{\mathbf T}_{w_{i_{2}}}\dots {\mathbf T}_{w_{i_{n}}}$$
and this expansion is independent of the chosen reduced decomposition.
\end{enumerate}
\end{proposition}
\noindent Now, we state and prove the main result of this section. This is an analogue of the results previously proven in \cite[Theorem 4.4.5]{Mac}, \cite[Proposition 7.3.7]{BGKP} and  \cite[Proposition 2.10]{MP2}.
\begin{proposition}\label{proprkgt1}
Let $\lambda^{\vee}$ be a dominant and regular, $w\in W$ be such that $w=w_{\alpha}w'$ and $l(w)=1+l(w')$. Then 
\begin{eqnarray}\label{Operator1}
I_{w,\lambda^{\vee}} &=&{\mathbf T}_{w_{\alpha}}(I_{w',\lambda^{\vee}}).\label{Operator1}
\end{eqnarray}
\end{proposition}

\noindent The above proposition has the following corollaries.
\begin{corollary}\label{deluttheorem}
For $w\in W$ and $\lambda\in\Lambda_{+}$ regular, the value of ${\mathbf T}_{w}(e^{\lambda^{\vee}})$ at $v=q^{-1}$ is equal to a constant multiple of the integral $I_{w,\lambda^{\vee}}.$ More precisely, 
\begin{eqnarray}\label{theoremeope}
I_{w,\lambda^{\vee}} &=&q^{-\langle \rho, \lambda^{\vee}\rangle}{\mathbf T}_{w}(e^{\lambda^{\vee}}).
\end{eqnarray}
\begin{proof}\let\qed\relax
Let $w=w_{\alpha_m}w_{\alpha_{m-1}}\dots w_{\alpha_1}$ be a reduced decomposition of $w$, then by Proposition~\ref{proprkgt1} 
\begin{eqnarray}
I_{w,\lambda^{\vee}} &=&{\mathbf T}_{w_{\alpha_{m}}}{\mathbf T}_{w_{\alpha_{m-1}}}\dots {\mathbf T}_{w_{\alpha_{2}}}(I_{w_{1},\lambda^{\vee}}).
\end{eqnarray}
The rank $1$ computation implies that $I_{w_{1},\lambda^{\vee}}=q^{-\langle \rho,\lambda^{\vee}\rangle}{\mathbf T}_{w_{1}}(e^{\lambda^{\vee}})$. Finally, this corollary follows by part (2) of Proposition~\ref{proplo}.
\end{proof}
\end{corollary}
 
\begin{corollary}\label{fsuppc2}
With the same assumptions as above, the integral $I_{w,\lambda^{\vee}}$ converges in the following sense: there exists a finite subset $\mathfrak{B}\subset \Lambda^{\vee}$ such that 
$$I_{w,\lambda^{\vee}}=\sum_{\mu^{\vee}\in\mathfrak{B}}c_{\mu^{\vee}}e^{\mu^{\vee}}$$
 with $c_{\mu^{\vee}}\in\mathbb{C}$ for all $\mu^{\vee}\in\mathfrak{B}$. 
 \begin{proof}
By Proposition~\ref{proprkgt1}, the assertion follows by showing ${\mathbf T}_{w}(e^{\lambda^{\vee}})\in \mathbb{C}[\Lambda^{\vee}]$ at $v=q^{-1}$. An affine version of this statement is given in \cite[Section 4.3]{Mac4} which extends to general Kac-Moody root systems as well.  
 \end{proof}
\end{corollary}
The way we defined Haar measure, it is easy to verify that there exists a constant $C>0$ such that 
\begin{eqnarray}\label{intcomparison}
\tilde{I}_{w,\lambda^{\vee}}=CI_{w,\lambda^{\vee}}.
\end{eqnarray}
Corollary~\ref{fsuppc2} and (\ref{intcomparison}) imply that the integral $\tilde{I}_{w,\lambda^{\vee}}$ also converges.
\noindent The proof of Proposition~\ref{proprkgt1} will be carried out in next two subsections.
\subsection{Rank 1 Proof}
 First, we prove the proposition in rank one where the computations are very similar to what Macdonald did in Proposition 4.3.1 from \cite{Mac1} for a slightly different integral.\\

\noindent {\bf Step 1: Decomposition }\\
The integral $I_{w_{\alpha},\lambda^{\vee}}$ can be split into two parts
\begin{eqnarray}
I_{w_{\alpha},\lambda^{\vee}}&=&\int_{U^{-}_{w_{\alpha},\pi}}\Phi_{\rho}(u^{-}_{w_{\alpha}}\pi^{w_{\alpha}\lambda^{\vee}})du^{-}_{w_{\alpha}}=I^{1}_{w_{\alpha},\lambda^{\vee}}-I^{2}_{w_{\alpha},\lambda^{\vee}},\label{diff}
\end{eqnarray}
where 
\begin{eqnarray}
I^{1}_{w_{\alpha},\lambda^{\vee}}&=&\int_{U_{-\alpha}(\mathcal{K})}\Phi_{\rho}(u_{-\alpha}\pi^{w_{\alpha}\lambda^{\vee}})du_{-{\alpha}},\\
 I^{2}_{w_{\alpha},\lambda^{\vee}}&=&\int_{U_{-\alpha}[\le 0]}\Phi_{\rho}(u_{-\alpha}\pi^{w_{\alpha}\lambda^{\vee}})du_{-{\alpha}},
 \end{eqnarray}
and where $U_{-\alpha}[\le 0]=\cup_{n\le 0}U_{-\alpha}[ n]$ and for $n\le 0$, 
\begin{eqnarray}
 U_{-\alpha}[ n]=\{u_{-\alpha}(t): val(t)=n\}.
\end{eqnarray}
\noindent {\bf Step 2: Evaluation of $I^{2}_{w_{\alpha},\lambda^{\vee}}$}\\
 We start by evaluating $I^{2}_{w_{\alpha},\lambda^{\vee}}$, which can be written as
\begin{eqnarray}\label{rank1piece1}
I^{2}_{w_{\alpha},\lambda^{\vee}}=\sum_{t=0}^{-\infty }\int_{U_{-\alpha}[t]}\Phi_{\rho}(u_{-\alpha}\pi^{w_{\alpha}\lambda^{\vee}})du_{-{\alpha}}.
\end{eqnarray}
Let $s\in\mathcal{K}$ with $val(s)=t$, $t\le 0$ and $s=\pi^{t}u$ for some $u\in \mathcal{O}^{\ast}$. We use the following identity which can be proven by using the relations (\ref{relationtorus1}) and (\ref{relationtorus2}).
\begin{eqnarray}
u_{-\alpha}(s)&=&u_{\alpha}(s^{-1})\pi^{-t\alpha^{\vee}}w_{\alpha}u_{\alpha}(s^{-1})
\end{eqnarray}
and write
\begin{eqnarray*}
u_{-\alpha}(s)\pi^{w_{\alpha}\lambda^{\vee}}&=&u_{\alpha}(s^{-1})\pi^{-t\alpha^{\vee}}w_{\alpha}u_{\alpha}(s^{-1})\pi^{w_{\alpha}\lambda^{\vee}}\\
&=&u_{\alpha}(s^{-1})\pi^{\lambda^{\vee}-t\alpha^{\vee}}w_{\alpha}u_{\alpha}(\pi^{\langle \alpha, -w_{\alpha}\lambda^{\vee}\rangle}s^{-1})\\
&=&u_{\alpha}(s^{-1})\pi^{\lambda^{\vee}-t\alpha^{\vee}}w_{\alpha}u_{\alpha}(\pi^{\langle \alpha, \lambda^{\vee}\rangle}s^{-1}).
\end{eqnarray*}
Since $\lambda^{\vee}$ is dominant and regular, $\langle \alpha, \lambda^{\vee}\rangle>0$. Therefore, 
\begin{eqnarray*}
\Phi_{\rho}(u_{-\alpha}\pi^{w_{\alpha}\lambda^{\vee}})&=&q^{-\langle\rho ,\lambda^{\vee}-t\alpha^{\vee}\rangle}e^{\lambda^{\vee}-t\alpha^{\vee}}\\
&=&q^{-\langle\rho ,\lambda^{\vee}\rangle}q^{t}e^{\lambda^{\vee}-t\alpha^{\vee}}.
\end{eqnarray*}
The integral in the sum (\ref{rank1piece1}) solves as,
\begin{eqnarray*}
\int_{U_{-\alpha}[t]}\Phi_{\rho}(u_{-\alpha}\pi^{w_{\alpha}\lambda^{\vee}})du_{-{\alpha}}&=&\int_{U_{-\alpha}[t]}q^{-\langle\rho ,\lambda^{\vee}\rangle}q^{t}e^{\lambda^{\vee}-t\alpha^{\vee}}du_{-{\alpha}}\\
&=&q^{-\langle\rho ,\lambda^{\vee}\rangle}q^{t}e^{\lambda^{\vee}-t\alpha^{\vee}}Vol(U_{-\alpha}[t])\\
&=&q^{-\langle\rho ,\lambda^{\vee}\rangle}q^{t}e^{\lambda^{\vee}-t\alpha^{\vee}}(q^{-t}-q^{-t-1})\\
&=&q^{-\langle\rho ,\lambda^{\vee}\rangle}e^{\lambda^{\vee}}(1-q^{-1})e^{-t\alpha^{\vee}}.
\end{eqnarray*}
Hence,
\begin{eqnarray}
I^{2}_{w_{\alpha},\lambda^{\vee}}&=&\sum_{t=0}^{-\infty }q^{-\langle\rho ,\lambda^{\vee}\rangle}e^{\lambda^{\vee}}(1-q^{-1})e^{-t\alpha^{\vee}}\nonumber\\
&=&q^{-\langle\rho ,\lambda^{\vee}\rangle}\frac{1-q^{-1}}{1-e^{\alpha^{\vee}}}e^{\lambda^{\vee}}.\label{I1a}
\end{eqnarray}

\noindent {\bf Step 3: Evaluation of $I^{1}_{w_{\alpha},\lambda^{\vee}}$}\\
Next, we compute $I^{1}_{w_{\alpha},\lambda^{\vee}}$. By the change of variables $u_{-\alpha}\mapsto \pi^{w_{\alpha}\lambda^{\vee}}u_{-\alpha}\pi^{-w_{\alpha}\lambda^{\vee}}$,
we get
\begin{eqnarray}
I^{1}_{w_{\alpha},\lambda^{\vee}}&=&\int_{U_{-\alpha}(\mathcal{K})}\Phi_{\rho}(u_{-\alpha}\pi^{w_{\alpha}\lambda^{\vee}})du_{-{\alpha}}\nonumber\\
&=&\jmath(u_{-\alpha}\pi^{w_{\alpha}\lambda^{\vee}}) \int_{U_{-\alpha}(\mathcal{K})}\Phi_{\rho}(\pi^{w_{\alpha}\lambda^{\vee}}u_{-\alpha})du_{-{\alpha}},\label{I1c}
\end{eqnarray}
where $\jmath(u_{-\alpha}\pi^{w_{\alpha}\lambda^{\vee}})$ is a {\it Jacobian} factor that is equal to $q^{-\langle\alpha ,\lambda^{\vee}\rangle}$. Now, we have two cases:\\
\noindent {\bf Case 1:} if $u_{-\alpha}\in U_{-\alpha}(\mathcal{O})$. Then $\Phi_{\rho}(\pi^{w_{\alpha}\lambda^{\vee}}u_{-\alpha})=q^{-\langle\rho ,w_{\alpha}\lambda^{\vee}-\rangle}e^{w_{\alpha}\lambda^{\vee}}$ and hence 
$$\int_{U_{-\alpha}(\mathcal{O})}\Phi_{\rho}(\pi^{w_{\alpha}\lambda^{\vee}}u_{-\alpha})du_{-{\alpha}}=q^{-\langle\rho ,w_{\alpha}\lambda^{\vee}\rangle}e^{w_{\alpha}\lambda^{\vee}} Vol(U_{-\alpha}(\mathcal{O}))=q^{-\langle\rho ,w_{\alpha}\lambda^{\vee}\rangle}e^{w_{\alpha}\lambda^{\vee}}.$$
\noindent {\bf Case 2:} If $u_{-\alpha}\notin U_{-\alpha}(\mathcal{O})$, then
\begin{eqnarray*}
\pi^{w_{\alpha}\lambda^{\vee}}u_{-\alpha}(s)&=&\pi^{w_{\alpha}\lambda^{\vee}}u_{\alpha}(s^{-1})\pi^{-val(s)\alpha^{\vee}}w_{\alpha}u_{\alpha}(s^{-1})\\
&=&u_{\alpha}(\pi^{-\langle\alpha ,\lambda^{\vee}\rangle}s^{-1})\pi^{w_{\alpha}\lambda^{\vee}-val(s)\alpha^{\vee}}w_{\alpha}u_{\alpha}(s^{-1}).
\end{eqnarray*}
So,\\
$\Phi_{\rho}(\pi^{w_{\alpha}\lambda^{\vee}}u_{-\alpha}(s))=q^{-\langle\rho ,w_{\alpha}\lambda^{\vee}-val(s)\alpha^{\vee}\rangle}e^{w_{\alpha}\lambda^{\vee}-val(s)\alpha^{\vee}}$
$=q^{-\langle\rho ,w_{\alpha}\lambda^{\vee}\rangle}q^{val(s)}e^{w_{\alpha}\lambda^{\vee}}e^{-val(s)\alpha^{\vee}}$
$=q^{-\langle\rho ,\lambda^{\vee}\rangle}q^{\langle\alpha ,\lambda^{\vee}\rangle}e^{w_{\alpha}\lambda^{\vee}}q^{val(s)}e^{-val(s)\alpha^{\vee}}.$\\
Putting the values of the function for both cases in (\ref{I1c}), we obtain
\begin{eqnarray}
I^{1}_{w_{\alpha},\lambda^{\vee}}&=&\jmath(u_{-\alpha}\pi^{w_{\alpha}\lambda^{\vee}})q^{-\langle\rho-\alpha ,\lambda^{\vee}\rangle}e^{w_{\alpha}\lambda^{\vee}}[1+(1-q^{-1})e^{\alpha^{\vee}}+(1-q^{-1})e^{2\alpha^{\vee}}+\dots]\nonumber\\
&=&q^{-\langle\rho ,\lambda^{\vee}\rangle}e^{w_{\alpha}\lambda^{\vee} } [1+\frac{(1-q^{-1})e^{\alpha^{\vee}}}{1-e^{\alpha^{\vee}}}]\nonumber\\
&=&q^{-\langle\rho ,\lambda^{\vee}\rangle}e^{w_{\alpha}\lambda^{\vee} } \frac{1-q^{-1}e^{\alpha^{\vee}}}{1-e^{\alpha^{\vee}}}. \label{I2b}
\end{eqnarray}
\\
\noindent {\bf Step 4: Conclusion}\\
Using (\ref{I2b}) and (\ref{I1a}) in (\ref{diff}), we get
\begin{eqnarray}
I_{w_{\alpha},\lambda^{\vee}}&=&q^{-\langle\rho ,\lambda^{\vee}\rangle}\frac{1-q^{-1}e^{\alpha^{\vee}}}{1-e^{\alpha^{\vee}}}e^{-w_{\alpha}\lambda^{\vee} } -q^{-\langle\rho ,\lambda^{\vee}\rangle}\frac{1-q^{-1}}{1-e^{\alpha^{\vee}}}e^{\lambda^{\vee}}\nonumber\\
&=&q^{-\langle\rho ,\lambda^{\vee}\rangle}[{\mathbf c}(\alpha^{\vee}_{i})e^{w_{i} \lambda^{\vee}}+{\mathbf  b}(\alpha^{\vee}_{i})e^{ \lambda^{\vee}}]\nonumber\\
&=&q^{-\langle\rho ,\lambda^{\vee}\rangle}{\mathbf  T}_{w_{\alpha_{i}}}(e^{\lambda^{\vee}}).\label{fin}
\end{eqnarray}

\subsection{Higher Rank Proof}
The assertion in higher rank is proven below.\\
\noindent {\bf Step 1: Preliminary Reduction:}\\
We will use the following description of the elements of $S^{-}_{w^{-1}}=\Delta^{-}\cap w\Delta^{+}$ which can be verified easily.
\begin{lemma}
Let $w\in W$ be such that $w=w_{\alpha}w'$ and $l(w)=1+l(w')$. Then
\begin{eqnarray}
S^{-}_{w^{-1}}&=&\{-\alpha\}\cup \{w_{\alpha}\beta\mid \beta \in S^{-}_{(w')^{-1}}\}.
\end{eqnarray}
\end{lemma}
\noindent The above lemma implies the following decomposition of $U^{-}_{w^{-1}}$.
\begin{lemma}\label{lem6.3}
By assuming the conditions on $w\in W$ from the above lemma, each $u^{-}_{w}\in U^{-}_{w^{-1}}$ can be written as $u^{-}_{w}=u_{-\alpha}w_{\alpha}u^{-}_{w'}w_{\alpha},$
where $u_{-\alpha}\in U_{-\alpha}$ and $u^{-}_{w'}\in U^{-}_{(w')^{-1}}$.
\end{lemma}
Lemma~\ref{lem6.3} yields the following splitting of the integral $I_{w,\lambda^{\vee}}$:
\begin{eqnarray*}
I_{w,\lambda^{\vee}}&=&\int_{U^{-}_{w^{-1},\pi}}\Phi_{\rho}(u^{-}_{w}\pi^{w\lambda^{\vee}})du^{-}_{w}\\
&=&\int_{U_{-\alpha,\pi}}\int_{U^{-}_{(w')^{-1},\pi}}\Phi_{\rho}(u_{-\alpha}w_{\alpha}u^{-}_{w'}w_{\alpha}\pi^{w\lambda^{\vee}})du_{-\alpha}du^{-}_{w'}\\
&=&I^{1}_{w,\lambda^{\vee}}-I^{2}_{w,\lambda^{\vee}},
\end{eqnarray*}
where
$$I^{1}_{w,\lambda^{\vee}}=\int_{U_{-\alpha}(\mathcal{K})}\int_{U^{-}_{(w')^{-1},\pi}}\Phi_{\rho}(u_{-\alpha}w_{\alpha}u^{-}_{w'}w_{\alpha}\pi^{w\lambda^{\vee}})du_{-\alpha}du^{-}_{w'}$$
and 
$$I^{2}_{w,\lambda^{\vee}}=\int_{U_{-\alpha}[\le 0]}\int_{(U^{-}_{w')^{-1},\pi}}\Phi_{\rho}(u_{-\alpha}w_{\alpha}u^{-}_{w'}w_{\alpha}\pi^{w\lambda^{\vee}})du_{-\alpha}du^{-}_{w'}.$$

\noindent {\bf Step 2: Evaluation of $I^{1}_{w,\lambda^{\vee}}$:}\\
 \noindent To simplify the integrand, we define a map.
\begin{definition} 
Let $A$ be the quotient group as introduced in Subsection~\ref{subdec}. The function
$$Iw_{A}:G\longrightarrow A$$
is defined by setting the formula $Iw_{A}(g)=\pi^{\mu^{\vee}}$ for all $g\in U\pi^{\mu^{\vee}}K$ and $\mu^{\vee}\in \Lambda^{\vee}$.
\end{definition}
Following Kumar \cite[P. 77]{Kum}, for a simple root $\alpha$, we denote a subset by $U^{\alpha}$ of $U^{+}$ which is equal to $w_{\alpha}U^{+}w_{\alpha}\cap U^{+}$. This subset is normalized by the root subgroups $U_{\alpha}$ and $U_{-\alpha}$ and each element $u$ of $U^{+}$ can be written as $u=u_{\alpha}u^{\alpha}$ for some $u_{\alpha}\in U_{\alpha}$ and $u^{\alpha}\in U^{\alpha}$. Writing $u^{-}_{w}\in U^{-}_{w^{-1},\pi}$ as in Lemma~\ref{lem6.3}, we have
 \begin{eqnarray}
u^{-}_{w}\pi^{w\lambda^{\vee}}&=&u_{-\alpha}w_{\alpha}u^{-}_{w'}w_{\alpha}\pi^{w\lambda^{\vee}}\nonumber\\
&=&u_{-\alpha}w_{\alpha}u^{-}_{w'}\pi^{w'\lambda^{\vee}}w_{\alpha},\label{Indic1}
\end{eqnarray}
for some $u_{-\alpha}\in U_{-\alpha}$ and $u^{-}_{w'}\in U^{-}_{(w')^{-1},\pi}$. Next assume $u^{-}_{w'}\pi^{w'\lambda^{\vee}}w_{\alpha}=u\pi^{\mu^{\vee}}k$ be an Iwasawa decomposition, $u=x_{\alpha}u^{\alpha}$ for some $x_{\alpha}\in U_{\alpha}$ and $u^{\alpha}\in U^{\alpha}$, and let $n_{-\alpha}\in U_{-\alpha}$ be defined as $n_{-\alpha}=w_{\alpha}x_{\alpha}w_{\alpha}$. The right hand side of (\ref{Indic1}) becomes
\begin{eqnarray}
u_{-\alpha}w_{\alpha}u^{-}_{w'}\pi^{w'\lambda^{\vee}}w_{\alpha}&=& u_{-\alpha}w_{\alpha}u\pi^{\mu^{\vee}}kw_{\alpha}\nonumber\\
&=& u^{\alpha}_{1}\pi^{w_{\alpha}\mu^{\vee}}(\pi^{-w_{\alpha}\mu^{\vee}}u_{-\alpha}n_{-\alpha}\pi^{w_{\alpha}\mu^{\vee}})w_{\alpha}kw_{\alpha}.\label{Indic2}
\end{eqnarray}
Let $\tilde{n}_{-\alpha}=\pi^{-w_{\alpha}\mu^{\vee}}u_{-\alpha}n_{-\alpha}\pi^{w_{\alpha}\mu^{\vee}}.$ Summarizing, we have
\begin{lemma}
In the above notations,
$$Iw_{A}(u^{-}_{w})=Iw_{A}(u^{-}_{w'})^{w_{\alpha}}Iw_{A}(\tilde{n}_{-\alpha}).$$
\end{lemma}
\noindent So, the integral $I^{1}_{w,\lambda^{\vee}}$ takes the form

$I^{1}_{w,\lambda^{\vee}}=\int_{U_{-\alpha}(\mathcal{K})}\int_{U^{-}_{(w')^{-1},\pi}}\Phi_{\rho}(Iw_{A}(u^{-}_{w'})^{w_{\alpha}}Iw_{A}(\tilde{n}_{-\alpha}))d\tilde{n}_{-\alpha}du^{-}_{w'}$

$=\int_{U_{-\alpha}(\mathcal{K})}\Phi_{\rho}(Iw_{A}(\tilde{n}_{-\alpha}))d\tilde{n}_{-\alpha}\int_{U^{-}_{(w')^{-1},\pi}}\Phi_{\rho}(Iw_{A}(u^{-}_{w'})^{w_{\alpha}})du^{-}_{w'}.$

The integral defined with measure $d\tilde{n}_{-\alpha}$ can be related to the integral defined with measure $du_{-\alpha}$ by a change of variables contain a Jacobian factor $q^{-\langle \alpha, \mu^{\vee}\rangle}$. So, we obtain
\begin{eqnarray*}
I^{1}_{w,\lambda^{\vee}}&=&\int_{U_{-\alpha}(\mathcal{K})}\Phi_{\rho}(Iw_{A}(\tilde{n}_{-\alpha}))d\tilde{n}_{-\alpha}\int_{U^{-}_{(w')^{-1},\pi}}\Phi_{\rho}(Iw_{A}(u^{-}_{w'})^{w_{\alpha}})du^{-}_{w'}\\
&=& \int_{U_{-\alpha}(\mathcal{K})}q^{-\langle \alpha, \mu^{\vee}\rangle}\Phi_{\rho}(Iw_{A}(u_{-\alpha}))d{u_{-\alpha}}\int_{U^{-}_{(w')^{-1},\pi}}q^{\langle \alpha, \mu^{\vee}\rangle}(\Phi_{\rho}(Iw_{A}(u^{-}_{w'})))^{w_{\alpha}}du^{-}_{w'}\\
&=&\int_{U_{-\alpha}(\mathcal{K})}\Phi_{\rho}(Iw_{A}(u_{-\alpha}))d{u_{-\alpha}}\int_{U^{-}_{(w')^{-1},\pi}}(\Phi_{\rho}(Iw_{A}(u^{-}_{w'})))^{w_{\alpha}}du^{-}_{w'}
\end{eqnarray*}
where in the second integral the following fact is used
$$\Phi_{\rho}(\pi^{w_{\alpha}\mu^{\vee}})=q^{-\langle \rho, w_{\alpha}\mu^{\vee}\rangle}e^{w_{\alpha}\mu^{\vee}}=q^{-\langle \rho-\alpha, \mu^{\vee}\rangle}e^{w_{\alpha}\mu^{\vee}}=q^{\langle \alpha, \mu^{\vee}\rangle}\Phi_{\rho}(\pi^{\mu^{\vee}})^{w_{\alpha}}.$$ The rank $1$ computation for the first integral now implies,
\begin{eqnarray}
I^{1}_{w,\lambda^{\vee}}&=&{\mathbf c}[\alpha^{\vee}]\int_{U^{-}_{(w')^{-1},\pi}}(\Phi_{\rho}(Iw_{A}(u^{-}_{w'})))^{w_{\alpha}}du^{-}_{w'}\nonumber\\
&=&{\mathbf c}[\alpha^{\vee}](I_{w',\lambda^{\vee}})^{w_{\alpha}}.\label{Ipart1}
\end{eqnarray}
\noindent {\bf Step 3: Evaluation of $I^{2}_{w,\lambda^{\vee}}$:}\\
For $t\in\mathcal{K}$ and $val(t)\le 0$, we write the integrand of $I^{2}_{w,\lambda^{\vee}}$ as\\
$u_{-\alpha}(t)w_{\alpha}u^{-}_{w'}w_{\alpha}\pi^{w\lambda^{\vee}}$\\
$=\pi^{w\lambda^{\vee}}(\pi^{-w\lambda^{\vee}}u_{-\alpha}(t)\pi^{w\lambda^{\vee}})(\pi^{-w\lambda^{\vee}}w_{\alpha}u^{-}_{w'}w_{\alpha}\pi^{w\lambda^{\vee}})$.\\
Thus
\begin{eqnarray}
u_{-\alpha}(t)w_{\alpha}u^{-}_{w'}w_{\alpha}\pi^{w\lambda^{\vee}}&=&\pi^{w\lambda^{\vee}}u_{-\alpha}(\pi^{\langle -\alpha,-w_{\alpha}w'\lambda^{\vee}\rangle } t)(w_{\alpha}\pi^{-w'\lambda^{\vee}}u^{-}_{w'}\pi^{w'\lambda^{\vee}}w_{\alpha})\nonumber\\
&=&\pi^{w\lambda^{\vee}}u_{-\alpha}(\pi^{-\langle \alpha,w'\lambda^{\vee}\rangle } t)(w_{\alpha}\pi^{-w'\lambda^{\vee}}u^{-}_{w'}\pi^{w'\lambda^{\vee}}w_{\alpha}).\label{cal1}
\end{eqnarray}
Set $n_{w'\alpha}=-\langle \alpha,w'\lambda^{\vee}\rangle $. Since $\lambda^{\vee}$ is dominant $n_{w'\alpha}$ is a non positive integer. We define the subset $U^{-}_{w'}[\lambda]\subset U^{-}_{w',\pi}$ by setting $U^{-}_{w'}[\lambda]=\pi^{-w'\lambda}U^{-}_{(w')^{-1},\pi}\pi^{w'\lambda}.$
We use (\ref{cal1}) and the above notation to write,
\begin{align}
&I^{2}_{w,\lambda^{\vee}}=\jmath_{1}\jmath_{2}\int_{U_{-\alpha}[\le n_{w'\alpha}]}\int_{U^{-}_{w'}[\lambda]}\Phi_{\rho}(\pi^{w\lambda^{\vee}}u_{-\alpha}w_{\alpha}u^{-}_{w'}w_{\alpha})du_{-\alpha}du^{-}_{w'}\nonumber\\
&=\jmath_{1}\jmath_{2}q^{-\langle \rho, w\lambda^{\vee}\rangle  }e^{w\lambda^{\vee}}\int_{U_{-\alpha}[\le n_{w'\alpha}]}\int_{U^{-}_{w'}[\lambda]}\Phi_{\rho}(u_{-\alpha}w_{\alpha}u^{-}_{w'}w_{\alpha})du_{-\alpha}du^{-}_{w'},\label{secintsol}
\end{align}
where $\jmath_{1}:=\jmath(u_{-\alpha}\pi^{w\lambda^{\vee}})=q^{-\langle w'\lambda,\alpha^{\vee}\rangle}\;\;\text{and}\;\; \jmath_{2}:=\jmath(u^{-}_{w'}\pi^{w'\lambda^{\vee}})$
are the Jacobian factors. Suppose
\begin{eqnarray}\label{modi}
J^{2}_{w,\lambda^{\vee}}:=\int_{U_{-\alpha}[\le n_{w'\alpha}]}\int_{U^{-}_{w'}[\lambda]}\Phi_{\rho}(\pi^{w\lambda^{\vee}}u_{-\alpha}w_{\alpha}u^{-}_{w'}w_{\alpha})du_{-\alpha}du^{-}_{w'}.
\end{eqnarray}
The following lemma will be used to write $J^{2}_{w,\lambda^{\vee}}$ as a product of two integrals.
\begin{lemma}
Let $u_{-\alpha}(t)\in U_{-\alpha}[\le n_{w'\alpha}]$ and $u^{-}_{w'}\in U^{-}_{w'}[\lambda]$, then\\
$Iw_{A}(u_{-\alpha}(t)w_{\alpha}u^{-}_{w'}w_{\alpha})=Iw_{A}(u_{-\alpha}(t)w_{\alpha})Iw_{A}( u_{-\alpha}(t^{-1})u^{-}_{w'}u_{-\alpha}(t^{-1})^{-1}).$
\proof Let $\tilde{u}^{-}_{w'}:=(u_{-\alpha}(t^{-1})u^{-}_{w'}u_{-\alpha}(t^{-1})^{-1})$. We have 
\begin{eqnarray}
u_{-\alpha}(t)w_{\alpha}u^{-}_{w'}w_{\alpha}\nonumber&=&u_{\alpha}(t^{-1})\pi^{-val(t)\alpha^{\vee}}u_{-\alpha}(t^{-1})u^{-}_{w'}w_{\alpha}\nonumber\\
&=&u_{\alpha}(t^{-1})\pi^{-val(t)\alpha^{\vee}}\tilde{u}^{-}_{w'}u_{-\alpha}(t^{-1})w_{\alpha}.\label{Iw2a}
\end{eqnarray}
Let $\tilde{u}^{-}_{w'}=u_{-\alpha}(t^{-1})u^{-}_{w'}u_{-\alpha}(t^{-1})^{-1}=u'\pi^{\nu^{\vee}}k',$ for some $u'\in U$ and $k'\in K$. Using this Iwasawa decomposition in (\ref{Iw2a}), we get
\begin{eqnarray}
u_{-\alpha}(t)w_{\alpha}u^{-}_{w'}w_{\alpha}&=&u_{\alpha}(t^{-1})\pi^{-val(t)\alpha^{\vee}}(u^{'}\pi^{\nu^{\vee}}k^{'})u_{-\alpha}(t^{-1})w_{\alpha}\nonumber\\
&=&u^{''}\pi^{\nu^{\vee}-val(t)\alpha^{\vee}}k^{'}u_{-\alpha}(t^{-1})w_{\alpha},\label{Iw2b}
\end{eqnarray}
for some $u^{''}\in U^{+}$. Thus 
$$u_{-\alpha}(t)w_{\alpha}u^{-}_{w'}w_{\alpha}\in U\pi^{\nu^{\vee}-val(t)\alpha^{\vee}}K$$
 and the assertion follows.
\end{lemma}
\noindent By following the above lemma integral $J^{2}_{w,\lambda^{\vee}}$ can be split,
\begin{eqnarray*}
J^{2}_{w,\lambda^{\vee}}=\int_{U_{-\alpha}[\le n_{w'\alpha}]}\Phi_{\rho}(u_{-\alpha}w_{\alpha})du_{-\alpha}\int_{U^{-}_{w'}[\lambda]}\Phi_{\rho}(u_{-\alpha}(t^{-1})u^{-}_{w'}u_{-\alpha}(-t^{-1}))du^{-}_{w'}.
\end{eqnarray*}
For $t\in \mathcal {K}$ with $val(t)=n_{w'\alpha}$, $u_{-\alpha}(t^{-1})\in U_{-\alpha}(\mathcal{O})$
therefore $U^{-}_{w'}[\lambda]$ and $u_{-\alpha}(t^{-1})U^{-}_{w'}[\lambda]u_{-\alpha}(-t^{-1})$ have the same measure and we can write
\begin{eqnarray*}
J^{2}_{w,\lambda^{\vee}}&=&\int_{U_{-\alpha}[\le n_{w'\alpha}]}\Phi_{\rho}(u_{-\alpha}w_{\alpha})du_{-\alpha}\int_{U^{-}_{w'}[\lambda]}\Phi_{\rho}(u^{-}_{w'})du^{-}_{w'}.\end{eqnarray*}
Since $\jmath_{1}=\jmath(u_{-\alpha}\pi^{w\lambda^{\vee}})=q^{-\langle w'\lambda,\alpha^{\vee}\rangle}$ and $n_{w'\alpha}=\langle \alpha, w'\lambda^{\vee}\rangle$, therefore\\
$\jmath_{1} q^{-\langle \rho, w\lambda^{\vee}\rangle  }e^{w\lambda^{\vee}}=\jmath_{1}q^{-\langle \rho, w_{\alpha}w'\lambda^{\vee}\rangle  }e^{w_{\alpha}w'\lambda^{\vee}}$
$=\jmath_{1}q^{-\langle \rho-\alpha, w'\lambda^{\vee}\rangle  }e^{w'\lambda^{\vee}-\langle \alpha, w'\lambda^{\vee}\rangle\alpha^{\vee} }$\\
$=\jmath_{1}q^{-\langle -\alpha, w'\lambda^{\vee}\rangle  }q^{-\langle \rho, w'\lambda^{\vee}\rangle  }e^{w'\lambda^{\vee}}e^{-\langle \alpha, w'\lambda^{\vee}\rangle\alpha^{\vee} }$
$=q^{-\langle \rho, w'\lambda^{\vee}\rangle  }e^{w'\lambda^{\vee}}e^{-n_{w'\alpha}\alpha^{\vee}}.$\\
Also,

$e^{-n_{w'\alpha}\alpha^{\vee}}\int_{U_{-\alpha}[\le n_{w'\alpha}]}\Phi_{\rho}(u_{-\alpha}w_{\alpha})du_{-\alpha}$\\
$=e^{-n_{w'\alpha}}[(1-q^{-1})e^{n_{w'\alpha}\alpha^{\vee}}+(1-q^{-1})e^{(n_{w'\alpha}+1)\alpha^{\vee}}+ \dots ]$\\
$=[(1-q^{-1})+(1-q^{-1})e^{\alpha^{\vee}}+ \dots ]$
$={\mathbf b}[\alpha^{\vee}].$

By putting these pieces back in (\ref{secintsol}), we obtain
\begin{eqnarray}
I^{1}_{w,\lambda^{\vee}}&=& \jmath_{2}{\bf b}[\alpha^{\vee}]\int_{U^{-}_{w'}[\lambda]}q^{-\langle \rho, w'\lambda^{\vee}\rangle  }e^{w'\lambda^{\vee}}\Phi_{\rho}(u^{-}_{w'}w_{\alpha})du^{-}_{w'}\nonumber\\
&=&\jmath_{2} {\mathbf b}[\alpha^{\vee}]\int_{U^{-}_{w'}}\Phi_{\rho}(\pi^{w'\lambda^{\vee}}u^{-}_{w'})du^{-}_{w'}\nonumber\\
&=&{\bf b}[\alpha^{\vee}]I_{w',\lambda^{\vee}}.\label{2ab}
 \end{eqnarray}
The solutions (\ref{Ipart1}) and (\ref{2ab}) imply that
\begin{eqnarray}
I_{w,\lambda^{\vee}} &=&{\mathbf T}_{w_\alpha}(I_{w',\lambda^{\vee}}).
\end{eqnarray}

\section{Finiteness of Fiber}\label{sectionfiberfin}
In this section, we fix $w\in\Omega$, $\mu^{\vee}\in\Lambda$ and $\lambda\in\Lambda_{+}$ regular, where $\Omega\subset W$ is the finite set obtained in Proposition~\ref{Wfin}. We start with the following terminology which will be used later in this section.

\begin{definition}
Let $f=\sum_{\mu^{\vee}\in \Lambda^{\vee}}c_{\mu^{\vee}}e^{\mu^{\vee}}$ be a formal sum, we write 
\begin{eqnarray}
[e^{\xi^{\vee}}]f:=c_{\xi^{\vee}}.
\end{eqnarray}
\end{definition}

\noindent Let $Z:=\{\mu^{\vee}\in \Lambda^{\vee}\mid m_{w,\lambda^{\vee}}^{-1}(\pi^{\mu^{\vee}})\ne \emptyset\}.$
\begin{lemma}\label{Convspf1}
For $w\in W$ and $\lambda^{\vee}\in \Lambda^{\vee}_{+}$ regular, $Z\subset supp(I_{w,\lambda^{\vee}})$, where
$supp(I_{w,\lambda^{\vee}})=\{\mu^{\vee}\in\Lambda^{\vee}\mid [e^{\mu^{\vee}}]I_{w,\lambda^{\vee}}\ne 0\}$.
\end{lemma}
\begin{proof}\let\qed\relax
Let $\mu^{\vee}$ is such that $m_{w,\lambda^{\vee}}^{-1}(\pi^{\mu^{\vee}})\ne \emptyset$, then
\begin{eqnarray*}
 wI^{-}\pi^{\lambda^{\vee}}K\cap U\pi^{\mu^{\vee}} K&\ne& \emptyset
  \end{eqnarray*}
  which implies
  \begin{eqnarray*}
 wU^{+}_{\pi}U^{-}_{\mathcal{O}}\pi^{\lambda^{\vee}}K\cap U\pi^{\mu^{\vee}} K&\ne&\emptyset.  
 \end{eqnarray*}
 Since $\lambda^{\vee}$ is dominant and regular, $\pi^{-\lambda^{\vee}}U^{-}_{\mathcal{O}}\pi^{\lambda^{\vee}}\subset K$ and this gives 
\begin{eqnarray*}
 U^{-}_{w^{-1}\pi}\pi^{w\lambda^{\vee}}\cap U\pi^{\mu^{\vee}} K\ne\emptyset .
\end{eqnarray*}
and thus $\mu^{\vee}\in supp(I_{w,\lambda^{\vee}})$.
\end{proof}

\subsection{Quotient Space and Surjection}
We equip the group $U^{+}_{w,\pi}$ with the following relation
\begin{definition}
Let $u_{w},z_{w}\in U^{+}_{w,\pi}$. We say $u_{w}\sim z_{w}$ if and only if 
\begin{eqnarray}
u_{w}= z_{w}\pi^{\lambda^{\vee}}U^{+}_{w,\pi}\pi^{-\lambda^{\vee}}.
\end{eqnarray}
\end{definition}
\noindent It can be easily verified that $\sim$ is an equivalence relation. For $u_{w}\in U^{+}_{w,\pi}$,  $[u_{w}]$ denotes the equivalence class of $u_{w}$ with respect to the relation $\sim$.
Next, set
\begin{eqnarray}
U^{+}_{w,\pi}(\mu^{\vee}):=\{u_{w}\in U^{+}_{w,\pi}\mid wu_{w}\pi^{\lambda^{\vee}}\in U\pi^{\mu^{\vee}}K\},
\end{eqnarray}
 and 
 \begin{eqnarray}
\mathcal{X}^{+}_{w,\pi}(\mu^{\vee}):=\{[u_{w}]\mid u_{w}\in U^{+}_{w,\pi}(\mu^{\vee})\}.
\end{eqnarray}
 If $[u_{w}]\in \mathcal{X}^{+}_{w,\pi}(\mu^{\vee})$ with $u_{w}\in U^{+}_{w,\pi}(\mu^{\vee})$, the relation 
\begin{eqnarray}
wu_{w}\pi^{\lambda^{\vee}}=u\pi^{\mu^{\vee}}k,
\end{eqnarray}
implies
\begin{eqnarray}
u^{-1}wu_{w}\pi^{\lambda^{\vee}}k^{-1}=\pi^{\mu^{\vee}},
\end{eqnarray}
for some $k\in K$ and $u\in U^{+}$. Thus $[u_{w}]\in \mathcal{X}^{+}_{w,\pi}(\mu^{\vee})$ gives rise to an element in $m_{w,\lambda^{\vee}}^{-1}(\mu^{\vee})$. 
\begin{definition}
Let 
\begin{eqnarray}
\phi\colon  \mathcal{X}^{+}_{w,\pi}(\mu^{\vee})\longrightarrow m_{w,\lambda^{\vee}}^{-1}(\mu^{\vee})
\end{eqnarray}
be a map defined as $\phi([u_{w}])\colon = (u^{-1}wu_{w}, \pi^{\lambda^{\vee}}k^{-1})$.
\end{definition}

\begin{lemma}\label{fibfinxlemma2}
The function $\phi$ is well defined and onto.
\proof To show that $\phi$ is a well defined, let $u_{w},z_{w}\in U^{+}_{w,\pi}(\mu^{\vee})$ and $u_{w}\sim z_{w}$ then there exists $u^{+}\in U^{+}_{w,\pi}$ such that 
\begin{eqnarray}
u_{w}= z_{w}\pi^{\lambda^{\vee}}u^{+}\pi^{-\lambda^{\vee}}.
\end{eqnarray}

\noindent Also, for some $k\in K$ and $u_{1}\in U^{+}$,
\begin{eqnarray}
wu_{w}\pi^{\lambda^{\vee}}=u_{1}\pi^{\mu^{\vee}}k,
\end{eqnarray}
which implies 
\begin{eqnarray}
\pi^{\mu^{\vee}}&=&u_{1}^{-1}wu_{w}\pi^{\lambda^{\vee}}k^{-1},\nonumber\\
&=&u_{1}^{-1}wu_{w}(\pi^{\lambda^{\vee}}u^{+}\pi^{-\lambda^{\vee}})^{-1}(\pi^{\lambda^{\vee}}u^{+}\pi^{-\lambda^{\vee}})\pi^{\lambda^{\vee}}k^{-1}\nonumber\\
&=&u_{1}^{-1}wz_{w}\pi^{\lambda^{\vee}}u^{+}k^{-1}\nonumber\\
&=&u_{1}^{-1}wz_{w}\pi^{\lambda^{\vee}}k',
\end{eqnarray}
where $u^{+}k^{-1}=k'\in K$. Thus, by taking $(\pi^{\lambda^{\vee}}u^{+}\pi^{-\lambda^{\vee}})^{-1}=i^{-}\in I^{-}$
\begin{eqnarray}
(u_{1}^{-1}wz_{w}, \pi^{\lambda^{\vee}}k')=(u_{1}^{-1}wu_{w}i^{-}, (i^{-})^{-1}\pi^{\lambda^{\vee}}k),
\end{eqnarray}
and hence 
\begin{eqnarray}
\phi([u_{w}])=\phi([z_{w}]).
\end{eqnarray}
Next, we show that $\phi$ is onto. Let $(x,y)\in m_{w,\lambda^{\vee}}^{-1}(\mu^{\vee})$ with $x=uwi^{-}_{1}$ and $y=i^{-}_{2}\pi^{\lambda^{\vee}}k$ for some $u\in U^{+}$, $k\in K$ and $i^{-}_{1},i^{-}_{2}\in I^{-}$. Then there exists $i^{-}\in I^{-}$ such that 
\begin{eqnarray}\label{fberfinx1}
uwi^{-}\pi^{\lambda^{\vee}}k=\pi^{\mu^{\vee}}.
\end{eqnarray}
Suppose $i^{-}$ has the following decomposition
\begin{eqnarray}
 i^{-}=u^{+}_{\pi}u^{-}_{\mathcal{O}}h_{\mathcal{O}}
 \end{eqnarray} 
for some $u^{+}_{\pi}\in U^{+}_{\pi}$, $u^{-}_{\mathcal{O}}\in U^{-}_{\mathcal{O}}$ and $h_{\mathcal{O}}\in H_{\mathcal{O}}$. Using it into (\ref{fberfinx1}), we have
\begin{eqnarray}\label{fberfin2}
\pi^{\mu^{\vee}}&=&uwu^{+}_{\pi}u^{-}_{\mathcal{O}}h_{\mathcal{O}}\pi^{\lambda^{\vee}}k\nonumber\\
&=&uwu_{w,\pi}\pi^{\lambda^{\vee}}k''
\end{eqnarray}
for some $u_{w,\pi}\in U^{+}_{w,\pi}$ and $k''=\pi^{-\lambda^{\vee}}u^{-}_{\mathcal{O}}h_{\mathcal{O}}\pi^{\lambda^{\vee}}k\in K$. So, we get an element $u_{w,\pi}\in U^{+}_{w,\pi}(\mu^{\vee})$ such that 
\begin{eqnarray}\label{fberfin3}
\phi([u_{w}])=(x,y).
\end{eqnarray}
This completes the proof.
\end{lemma}
\subsection{Finiteness of Level Sets }
Let $U^{-}_{w^{-1} }[\lambda^{\vee}]\colon=\pi^{-w\lambda^{\vee}}U^{-}_{w^{-1},\pi}\pi^{w\lambda^{\vee}}\;\;\text{and}\; \; U^{-}_{w^{-1}}[\lambda^{\vee},\mu^{\vee}]:= U^{-}_{w^{-1}}[\lambda^{\vee}]\cap U\pi^{\mu^{\vee}}K.$

\begin{remark}
If $u_{w}\in U^{+}_{w,\pi}$ satisfies $wu_{w}\pi^{\lambda^{\vee}}\in U\pi^{\mu^{\vee}}K$, then 
\begin{eqnarray}
\pi^{-w\lambda^{\vee}}u_{w}\pi^{w\lambda^{\vee}}\in U\pi^{\mu^{\vee}-w\lambda^{\vee}}K\cap U^{-}_{w^{-1} }[\lambda^{\vee}].
\end{eqnarray}
\end{remark}
\noindent Set
 $$\widehat{Y}:=\{\mu^{\vee}-w\lambda^{\vee}\mid \mu^{\vee}\in Y\},$$ 
where $Y=Supp(\tilde{I}_{w,\lambda^{\vee}})$.

\begin{lemma}\label{lemfiberfin1}
For each $\xi\in\widehat{Y}$, the coset space $ U^{-}_{w^{-1}}[\lambda^{\vee},\xi^{\vee}]/U^{-}_{w^{-1},\pi}$ is finite.
\proof The integral $\tilde{I}_{w,\lambda^{\vee}}$ which is defined in Section~\ref{Int} can be written as
\begin{eqnarray}
 \tilde{I}_{w,\lambda^{\vee}}&=&\int_{U^{-}_{w^{-1},\pi}}\Phi_{\rho}(u^{-}_{w}\pi^{w\lambda^{\vee}})\tilde{du^{-}_{w}}\nonumber\\
 &=&\int_{U^{-}_{w^{-1}}[\lambda^{\vee}]}\Phi_{\rho}(\pi^{w\lambda^{\vee}}u^{-}_{w})\tilde{du^{-}_{w}}\nonumber\\
 &=&q^{-\langle \rho, w\lambda^{\vee}\rangle  }e^{w\lambda^{\vee}}\int_{U^{-}_{w^{-1}}[\lambda^{\vee}]}\Phi_{\rho}(u^{-}_{w})\tilde{du^{-}_{w}}\nonumber\\
 &=&q^{-\langle \rho, w\lambda^{\vee}\rangle  }e^{w\lambda^{\vee}}\sum_{\xi^{\vee}\in \widehat{Y}}Vol(U^{-}_{w^{-1}}[\lambda^{\vee}]\cap U\pi^{\xi^{\vee}}K) q^{-\langle \rho, \xi^{\vee}\rangle }e^{\xi^{\vee}}\nonumber\\
 &=&q^{-\langle \rho, w\lambda^{\vee}\rangle  }e^{w\lambda^{\vee}}\sum_{\xi^{\vee}\in \widehat{Y}}|U^{-}_{w^{-1}}[\lambda^{\vee}]\cap U\pi^{\xi^{\vee}}K/U^{-}_{w^{-1},\pi}| q^{-\langle \rho, \xi^{\vee}\rangle }e^{\xi^{\vee}}\nonumber\\
 &=&\sum_{\xi^{\vee}\in \widehat{Y}}|U^{-}_{w^{-1}}[\lambda^{\vee}]\cap U\pi^{\xi^{\vee}}K/U^{-}_{w^{-1},\pi}| q^{-\langle \rho, w\lambda^{\vee}+\xi^{\vee}\rangle }e^{w\lambda^{\vee}+\xi^{\vee}}.\label{quotientfiberfin}
 \end{eqnarray} 
By Theorem~\ref{deluttheorem}, for each $\mu^{\vee}\in Y$, there exists a constant $D$ such that
\begin{eqnarray}\label{fberfin5}
[e^{\mu^{\vee}}] \tilde{I}_{w,\lambda^{\vee}}=D[e^{\mu^{\vee}}] T_{w}(e^{\lambda^{\vee}}).
\end{eqnarray}
Since the right hand side of (\ref{fberfin5}) is finite, so is the left hand side and hence the lemma follows.
\end{lemma}

\subsection{Main Results}
In this subsection, we obtain the finiteness of $m_{w,\lambda^{\vee}}^{-1}(\mu^{\vee})$. 

\begin{lemma}\label{fibfinprop1}
There exists a one-to-one map from $\mathcal{X}^{+}_{w,\pi}(\mu^{\vee})$ to $U^{-}_{w^{-1}}[\lambda^{\vee},\xi^{\vee}]/U^{-}_{w^{-1},\pi}$.
\proof 

 By using the fact
$$w\pi^{-\lambda}U^{+}_{w.\pi}(\mu^{\vee})\pi^{\lambda}w^{-1}\subset U^{-}_{w^{-1}}[\lambda^{\vee},\xi^{\vee}]$$
\noindent we define a map 
\begin{eqnarray*}
\psi\colon \mathcal{X}^{+}_{w,\pi}(\mu^{\vee})\longrightarrow U^{-}_{w^{-1}}[\lambda^{\vee},\xi^{\vee}]/U^{-}_{w^{-1},\pi}
\end{eqnarray*}
as,
\begin{eqnarray}
\psi ([u_{w}])=(w\pi^{-\lambda}u_{w}\pi^{\lambda}w^{-1})U^{-}_{w^{-1},\pi}.
\end{eqnarray}
We prove that $\psi$ is our required one-to-one map. First, we show that \\

\noindent {\bf $\psi$ is well defined:} Let $u_{w},z_{w}\in U^{+}_{w,\pi}(\mu^{\vee})$ and $u_{w}\sim z_{w}$ then there exists $u^{+}\in U^{+}_{w,\pi}$ such that 
\begin{eqnarray}
u_{w}= z_{w}\pi^{\lambda^{\vee}}u^{+}\pi^{-\lambda^{\vee}}.
\end{eqnarray}
Hence $(z_{w})^{-1}u_{w}= \pi^{\lambda^{\vee}}u^{+}\pi^{-\lambda^{\vee}}$ and 
\begin{eqnarray}
w\pi^{-\lambda^{\vee}}(z_{w})^{-1}u_{w}\pi^{\lambda^{\vee}}w^{-1}=w\pi^{-\lambda^{\vee}}(\pi^{\lambda^{\vee}}u^{+}\pi^{-\lambda^{\vee}})\pi^{\lambda^{\vee}}w^{-1}.
\end{eqnarray}
Since $u^{-}=w(u^{+})^{-1}w^{-1}\in U^{-}_{w^{-1},\pi}$, we get
\begin{eqnarray}
w\pi^{-\lambda^{\vee}}u^{-}_{w^{-1}}\pi^{\lambda^{\vee}}w^{-1}U^{-}_{w,\pi}&=& w\pi^{-\lambda^{\vee}}z^{-}_{w^{-1}}\pi^{\lambda^{\vee}}w^{-1}U^{-}_{w,\pi}
\end{eqnarray}
and hence $\psi$ is well defined.\\

{\bf $\psi$ is injective:} To show that $\psi$ is one-one, suppose  $u_{w},z_{w}\in U^{+}_{w,\pi}(\mu^{\vee})$ be such that
$$\psi([u_{w}])=\psi([z_{w}])$$
that is,
\begin{eqnarray}
w\pi^{-\lambda^{\vee}}z_{w}^{-1}u_{w}\pi^{w\lambda^{\vee}}w^{-1}&\in& U^{-}_{w^{-1},\pi}\nonumber\\
\pi^{-\lambda^{\vee}}z_{w}^{-1}u_{w}\pi^{\lambda^{\vee}} &\in& w^{-1} U^{-}_{w^{-1},\pi}w.\nonumber
\end{eqnarray}
Consequently,  $z_{w}^{-1}u_{w}\in \pi^{\lambda^{\vee}} U^{+}_{w, \pi} \pi^{-\lambda^{\vee}}$ and $z_{w} \sim u_{w}$. Hence $\psi$ is a one-one map. 
\end{lemma}

\begin{proposition}\label{finalprop}
For $w\in \Omega$, $\mu^{\vee}\in  \Lambda^{\vee}$ and $\lambda^{\vee}\in \Lambda^{\vee}_{+}$ regular, the fibers $m_{w,\lambda^{\vee}}^{-1}(\pi^{\mu^{\vee}})$ is finite.
\proof By Lemma~\ref{fibfinprop1}, the set $ \mathcal{X}^{+}_{w,\pi}(\mu^{\vee})$ is embedded in $U^{-}_{w^{-1}}[\lambda^{\vee},\xi^{\vee}]/U^{-}_{w^{-1},\pi}$. The quotient  $U^{-}_{w^{-1}}[\lambda^{\vee},\xi^{\vee}]/U^{-}_{w^{-1},\pi}$ is finite  by Lemma~\ref{lemfiberfin1} and hence $ \mathcal{X}^{+}_{w,\pi}(\mu^{\vee})$ is finite. Finally, by Lemma~\ref{fibfinxlemma2} the finite set $ \mathcal{X}^{+}_{w,\pi}(\mu^{\vee})$ is mapped onto the fiber $m_{w,\lambda^{\vee}}^{-1}(\mu^{\vee})$ which implies the finiteness of $m_{w,\lambda^{\vee}}^{-1}(\mu^{\vee})$. 
\end{proposition}

This completes the proof of the Weak Spherical Finiteness.

\section{ Finiteness Results}\label{secapplication}
In this section, we get the proofs the other finiteness theorems.
\subsection{Proof of Gindikin-Karpelevich Finiteness}In the following, we apply the weak Spherical Finiteness and Approximation Theorem to get the proof of the Gindikin-Karpelevich Finiteness.

\begin{proof}\let\qed\relax
By the Approximation Theorem, for $\mu^{\vee}\in \Lambda^{\vee}$ there exists $\lambda^{\vee}\in \Lambda^{\vee}_{+}$ regular and sufficiently big such that 
\begin{eqnarray}\label{eqsphfinw1}
K\pi^{\lambda^{\vee}}K\cap K\pi^{\mu^{\vee}}U^{+} =K\pi^{\lambda^{\vee}}U^{-}\cap K\pi^{\mu^{\vee}}U^{+}.
\end{eqnarray}
The Weak Spherical Finiteness implies that by choosing $\lambda^{\vee}$ sufficiently dominant, the set $K\backslash K\pi^{\lambda^{\vee}}K\cap K\pi^{\mu^{\vee}}U^{+} $ is finite and hence by (\ref{eqsphfinw1}), $K\backslash K\pi^{\lambda^{\vee}}U^{-}\cap K\pi^{\mu^{\vee}}U^{+}$ is also finite. Now, there is a bijection of sets 
\begin{eqnarray}
 K\pi^{\lambda^{\vee}}U^{-}\cap K\pi^{\mu^{\vee}}U^{+}&=&  K\pi^{\lambda^{\vee}}U^{-}\pi^{-\lambda^{\vee}}\cap K\pi^{\mu^{\vee}}U^{+}\pi^{-\lambda^{\vee}}\nonumber\\
 &=&KU^{-}\cap K\pi^{\mu^{\vee}-\lambda^{\vee}}U^{+}.
\end{eqnarray}
So, $K\backslash KU^{-}\cap K\pi^{\xi^{\vee}}U^{+}$ is finite for any $\xi^{\vee}=\mu^{\vee}-\lambda^{\vee}\in \Lambda^{\vee}$. This implies the Gindikin-Karpelevich Finiteness for $\lambda^{\vee}=0$, and hence the finiteness follows in general as well.
\end{proof}
\subsection{Proof of Spherical Finiteness}
Lastly, we sketch the proof of Spherical Finiteness as an implication of the Gindikin-Karpelevich Finiteness which follows exactly as it does in affine case (see \cite[Section 6.2]{BGKP}).

\begin{proof}\let\qed\relax
Let $\xi^{\vee}\in \Lambda^{\vee}$ and $\lambda^{\vee}\in \Lambda^{\vee}_{+}$. Theorem~\ref{maintheorem2} and the following fact from Theorem 1.9 of {\it op. cit.} (which also follows in general) 
\begin{eqnarray}\label{eqsphfin}
K\pi^{\xi^{\vee}}U^{-}\cap  K\pi^{\lambda^{\vee}}K=\emptyset\;\;\text{unless}\;\; \xi^{\vee}\le \lambda^{\vee},
\end{eqnarray}
allow us to write
 \begin{eqnarray}\label{eqsphfin}
K\pi^{\lambda^{\vee}}K\cap K\pi^{\mu^{\vee}}U^{+} =\bigcup_{\mu^{\vee}\le \xi^{\vee}\le \lambda^{\vee}} K\pi^{\xi^{\vee}}U^{-}\cap K\pi^{\mu^{\vee}}U^{+}\cap K\pi^{\lambda^{\vee}}K.
\end{eqnarray}
The Spherical Finiteness, i.e., the finiteness of $K\backslash K\pi^{\lambda^{\vee}}K\cap K\pi^{\mu^{\vee}}U^{+}$ is then a consequence of the finiteness of 
the set $\mathcal{B}_{\lambda^{\vee},\mu^{\vee}}=\{\xi^{\vee}\in\Lambda^{\vee}\mid \mu^{\vee}\le \xi^{\vee}\ \le\lambda^{\vee} \},$
 containment
\begin{eqnarray}
 K\pi^{\xi^{\vee}}U^{-}\cap K\pi^{\mu^{\vee}}U^{+}\cap K\pi^{\lambda^{\vee}}K\subset K\pi^{\xi^{\vee}}U^{-}\cap K\pi^{\mu^{\vee}}U^{+}
 \end{eqnarray}
 and Gindikin-Karpelevich Finiteness.  
 \end{proof}

\end{document}